\documentclass[11pt]{article}
\usepackage{graphicx} 

\newcommand{\bz}{{\mathbf z}}

\newcommand{\bv}{{\boldsymbol v}}

\newcommand{\Var}{\mathrm{Var}}

\newcommand{\Bcal}{\mathcal{B}}

\newcommand{\Ical}{\mathcal{I}}
\newcommand{\Lcal}{\mathcal{L}}

\newcommand{\Ncal}{\mathcal{N}}
\newcommand{\Ocal}{\mathcal{O}}

\newcommand{\Ucal}{\mathcal{U}}

\newcommand{\Ycal}{\mathcal{Y}}

\newcommand{\Ebb}{\mathbb{E}}

\newcommand{\Jbb}{\mathbb{J}}
\newcommand{\Nbb}{\mathbb{N}}
\newcommand{\Pbb}{\mathbb{P}}

\newcommand{\Rbb}{\mathbb{R}}

\newcommand{\Sbb}{\mathbb{S}}

\newcommand{\bx}{{\boldsymbol x}}

\newcommand{\blambda}{{\boldsymbol \lambda}}

\newcommand{\by}{{\boldsymbol y}}
\newcommand{\bg}{{\boldsymbol g}}
\newcommand{\bh}{{\boldsymbol h}}
\newcommand{\bu}{{\boldsymbol u}}

\newcommand{\one}{\mathbbm{1}}
\newcommand{\oneb}{\boldsymbol{\mathbbm{1}}}
\newcommand{\onebf}{\boldsymbol{\mathbf{1}}}

\newcommand{\dE}{\mathbb{E}}
\newcommand{\dP}{\mathbb{P}}

\newcommand{\cN}{\mathcal{N}}

\usepackage[margin=1in]{geometry}

\usepackage[utf8]{inputenc} 
\usepackage{url}            
\usepackage{booktabs}       
\usepackage{amsfonts}       
\usepackage{nicefrac}       
\usepackage{microtype}      
\usepackage{xcolor}
\usepackage{hyperref}
\hypersetup{%
  colorlinks=true,
  linkcolor=blue,
  citecolor=blue
}
\usepackage{bbm}
\usepackage{amsthm}
\usepackage{amsmath}        
\usepackage{amssymb}        
\usepackage{comment}        
\usepackage{enumitem}  
\usepackage{parskip}
\usepackage{tikz}
\usepackage{xspace}
\usepackage{subcaption}
\usepackage{mathtools}
\usepackage{amsthm}
\usepackage{thmtools,thm-restate}
\usetikzlibrary{calc}
\newtheorem{theorem}{Theorem}[section]
\newtheorem*{theorem*}{Theorem}
\newtheorem{proposition}[theorem]{Proposition}
\newtheorem{definition}[theorem]{Definition}
\newtheorem{remark}[theorem]{Remark}
\newtheorem*{remark*}{Remark}
\newtheorem{lemma}[theorem]{Lemma}

\DeclareMathOperator*{\argmax}{arg\,max}
\DeclareMathOperator*{\argmin}{arg\,min}

\title{The value of random zero-sum games
}
\date{January 2026}
\author{Romain Cosson \\ {\small \texttt{romain.cosson@nyu.edu}} \and Laurent Massoulié \\ {\small \texttt{laurent.massoulié@inria.fr}}}
\begin{document}

\maketitle
\begin{abstract}
    We study the value of a two-player zero-sum game on a random matrix $M\in \Rbb^{n\times m}$, defined by $v(M) = \min_{\bx\in\Delta_n}\max_{\by\in \Delta_m}\bx^TM \by$. In the setting where $n=m$ and $M$ has i.i.d. standard Gaussian entries, we prove that the standard deviation of $v(M)$ is of order $\frac{1}{n}$. This confirms an experimental conjecture dating back to the 1980s. 
    We also investigate the case where $M$ is a rectangular Gaussian matrix with $m = n+\lambda\sqrt{n}$, showing that the expected value of the game is of order $\frac{\lambda}{n}$, as well as the case where $M$ is a random orthogonal matrix.
    Our techniques are based on probabilistic arguments and convex geometry. 
    We argue that the study of random games could shed a new light on various problems in theoretical computer science.   
\end{abstract}

\thispagestyle{empty}

\clearpage
\thispagestyle{empty}

\tableofcontents
\clearpage
\pagenumbering{arabic}

\section{Introduction}
In a 1928 paper \cite{v1928theorie} titled ``On the theory of strategic games'',  John von Neumann established that for any real matrix $M\in \Rbb^{n\times m}$,
\begin{equation}\label{eq:vonNeumann}
   v(M) = \min_{\bx\in \Delta_n} \max_{\by\in \Delta_m} \bx^T M \by = \max_{\by\in \Delta_m} \min_{\bx\in \Delta_n} \bx^TM\by  
\end{equation}
where $\Delta_n=\{\bx\in \Rbb_+^n : \sum_i x_i = 1\}$ is the $n$-dimensional probability simplex and where $v(M)$ is called the value of the two-player zero-sum game defined by $M$. This influential result is known as von Neumann's minimax theorem.

The connection to game theory is as follows. 
Consider the strategic game where player~$1$ (resp. player $2$) chooses a row $i\in [n]$ (resp. a column $j\in [m]$) of a matrix $M\in\Rbb^{n\times m}$, and where the payoff $M_{ij}$ is transferred from player $1$ to player $2$. In this game, it is clearly valuable for a player to know the strategy of their opponent before they commit to their own choice. Yet, von Neumann's minimax theorem implies that if the players are allowed to use so-called \textit{mixed strategies}---probability distributions over rows (resp. columns)---the order in which they play does not influence the value of the game. Mathematically, this discussion is summarized as follows:
\begin{equation}\label{eq:saddle}
   \max_{j\in [m]} \min_{i\in [n]} M_{ij} \leq \max_{\by\in \Delta_m} \min_{i\in [n]}(M\by)_i = v(M) =  \min_{\bx\in \Delta_n} \max_{j\in [m]} (\bx^TM)_j \leq \min_{i\in [n]} \max_{j\in [m]} M_{ij}.
\end{equation}
The quantity $v(M)$ stands out as a fundamental quantity associated with the matrix $M$, and plays a central role in the analysis of algorithms. In this paper, we study $v(M)$ when $M$ is a random matrix, continuing a line of work that began in the 1960s. The words of Thomas M. Cover, in 1966, still very much apply today: \textit{``Even in the
simplest cases, determining the distribution of the value of a random game is
largely an unsolved problem''} \cite{cover1966probability}.

\paragraph{Main results.}  We apply modern techniques in probability and convex geometry to improve our understanding of the distribution of $v(M)$. Specifically, we provide tail bounds for $v(M)$ when $M\in\Rbb^{n\times m}$ is a random matrix with i.i.d. standard Gaussian entries and $m/2\leq n \leq 2m$.
For square matrices, our results imply that the standard deviation of the value of a random game is $\Theta(1/n)$, validating an experimental conjecture due to Faris and Maier \cite{faris1987value} and improving upon the previous $\Ocal(1/\sqrt{n})$ bound of Jonasson \cite{jonasson2004optimal}. 
For rectangular matrices, we show that there is a competition between the noise and the aspect ratio of the matrix in the regime where $m = n +\lambda \sqrt{n}$. Specifically, in that regime, we show that the expectation of $v(M)$ is of order $\lambda/n$ with fluctuations of $v(M)$ being of order $1/n$. We then study the case where $M\in \Rbb^{n\times n}$ is a random orthogonal matrix, providing tail bounds for $v(M)$ that directly imply a standard deviation in $\Ocal(1/n^{3/2})$. For both Gaussian and orthogonal matrices, we bring new powerful tools to the study of random games, which we outline in Section \ref{sec:approach-comp} (for the Gaussian part) and in Section \ref{sec:background-convex} (for the orthogonal part). 

\paragraph{Connection to theoretical computer science.} In the analysis of algorithms and decision rules, von Neumann's minimax theorem is more commonly known as Yao's minimax principle, in reference to the computer scientist Andrew Yao. The principle states that, for a given algorithmic problem, the worst-case performance of the best randomized algorithm is equal to the performance over the worst input distribution of the best deterministic algorithm. 
The principle is a direct application of Equation \eqref{eq:vonNeumann} to the matrix $M$ with rows corresponding to algorithms and columns corresponding to inputs, where the cell $M_{ij}$ corresponds to the cost of algorithm $i\in[n]$ on input $j\in [m]$, and where the value $v(M)$ corresponds to the complexity of the algorithmic problem represented by $M$. 
The first application of the principle by Yao \cite{yao1977probabilistic} was a proof that any randomized algorithm testing the existence of a perfect matching in a graph requires $\Omega(|V|^2)$ probes in expectation, where $|V|$ is the number of nodes in the graph. The framework is now widespread, appearing across various fields such as online algorithms, communication complexity, and computational complexity (see examples in \cite{motwani1996randomized,borodin2005online,wigderson2019mathematics}). In the last section of the paper, we discuss the use of random matrices to explore the magnitude of the advantage that randomization confers in worst-case analysis, and we highlight open questions that we believe relevant to the theoretical computer science community.

\paragraph{Relation to statistical physics.} A central aim of statistical physics is to understand the typical behavior of large and complex systems by studying probabilistic models of their microscopic structure. This point of view has proved remarkably fruitful for identifying emergent phenomena, such as phase transitions, in which a macroscopic property changes abruptly as a parameter crosses a threshold. Classic examples include percolation on random graphs \cite{bollobas2011random} where a giant component appears at average degree one, and the satisfiability threshold in random $k$-SAT \cite{mertens2006threshold}.

Random games are a natural counterpart of such models in the game-theoretic setting. Here, the “microscopic” randomness lies in the entries of the payoff matrix, which represent noisy or heterogeneous interactions between strategies. The “macroscopic” observable is the game value $v(M)$. From a statistical physics perspective, $v(M)$ plays a role analogous to the ground-state energy of a disordered system---it is the outcome of a global optimization (minimization over one player’s strategy, maximization over the other’s) on a random energy landscape. This parallel suggests that random games may exhibit their own forms of “game-theoretic phase transitions”, where, for example, the sign of $v(M)$ changes sharply as the aspect ratio $m/n$ crosses $1$. This particular phenomenon was formally proven by Cover \cite{cover1966probability}.  

As we will see, the toolkit developed in modern statistical physics is suited to the study of random games. Specifically, we find that Gaussian comparison inequalities \cite{gordon1985some,ledoux2013probability}---which have been used to analyze phase transitions in statistical settings \cite{miolane2021distribution}---are applicable in this context. 
Our work thus fits into the broader program of leveraging ideas from statistical physics and probability to understand scaling laws of algorithmic problems (see, e.g., \cite{decelle2011asymptotic}). In the same way that random graphs became canonical objects in the study of typical-case behavior (e.g., spectral algorithms for clustering \cite{abbe2018community}), we argue that random games merit attention and form a clean, analytically tractable model at the interface of probability, optimization, and game theory. For further connections to statistical physics, we refer to \cite{berg1998matrix,berg2000statistical} who approach the problem of random games with the so-called `replica' heuristic in the regime where $m=\alpha n$ for some constant $\alpha\in \Rbb$.

\paragraph{Previous works.}
Starting in the 1960s, with the work of Chernoff and Teicher \cite{chernoff1965limit}, there has been a limited but steady line of work studying random two-player zero-sum games. Initially \cite{chernoff1965limit} studied the value of pure strategies, describing the class of limiting distributions for $\min_{i\in [n]} \max_{j\in [m]}M_{ij}$. 
This is unfortunately of little help in the study of $v(M)$, because the inequalities $\max_{j\in [m]}\min_{i\in [n]} M_{ij}\leq v(M)\leq \min_{i\in [n]} \max_{j\in [m]} M_{ij}$ are typically very loose. In the special case where these inequalities are tight (i.e., conditioned on the unlikely event that the matrix $M$ has a saddle point) Thrall and Falk \cite{thrall1965some} explicitly derived the distribution of $v(M)$. 
The first general result on the value of $v(M)$ in a random setting analog to ours was by Cover \cite{cover1966probability} who showed that $\Pbb(v(M)> 0) = \Pbb(X\leq m)$ where $X$ is a binomial random variable with parameters $X\sim \Bcal(1/2,n+m-1)$. This result implies that $v(M)$ is almost always positive when $m>>n+\sqrt n$, which is consistent with the tail bounds provided in this paper. The proof of Cover uses a geometric argument, attributed to Schläfli \cite{schlafli1950gesammelte}, on the probability that a random subspace of dimension $n$ in $\Rbb^m$ intersects the positive orthant, and has the notable advantage of being applicable to a wider set of distributions. We discuss their work in more detail in Section \ref{sec:orth}, where we explain our own geometric argument. 
Later, in the 1980s, despite the lack of theoretical progress on the study of the value of $v(M)$, an empirical work of Faris and Maier \cite{faris1987value} conjectured that $|v(M)|$ is typically of order $\Theta(1/n)$ for $M\in\Rbb^{n\times n}$ with i.i.d. Gaussian entries. They also noticed that the cardinality of the support of optimal strategies is very close to a Binomial distribution $\Bcal(1/2,n)$. More recently, Jonasson \cite{jonasson2004optimal} proved that the cardinality of the support of optimal mixed strategies is within $[0.1n,0.9n]$ with high probability. 
He also proved that $|v(M)|$ is of order $O\left(1/\sqrt{n}\right)$ with high probability, slightly improving over the naive $O\left(\sqrt{\log n}/\sqrt{n}\right)$ bound (see Section \ref{sec:prelim})---and highlighted the important question of showing that the game value is effectively in $\Theta\left(1/n\right)$ (Conjecture 2.3 in \cite{jonasson2004optimal}).

We note that random $N$-player games \cite{rinott2000number,mclennan2005expected,amiet2021pure,heinrich2023best}, random bimatrix games \cite{mclennan2005asymptotic,alon2021dominance}, and random zero-sum games \cite{roberts2006nash,brandl2017distribution} are also studied in the economics literature. Specifically, Roberts \cite{roberts2006nash} investigated the game value for a heavy-tailed (Cauchy) distribution, and Brandl \cite{brandl2017distribution}, showed that for random skew-symmetric matrices, the support of the optimal strategy of the corresponding zero-sum game is uniformly distributed on the subsets of $[n]$ that are of odd cardinality. Random zero-sum games have also been studied in biology \cite{cohen1989host}, mathematics~\cite{tanikawa2010stability}, and statistical physics \cite{berg1998matrix,berg2000statistical}.

\paragraph{Organization of the paper.} The paper is organized as follows. First, Section \ref{sec:prelim} reviews preliminary results and definitions that will be used throughout the paper. Then, Section \ref{sec:main} provides our main result for Gaussian matrices. This section relies on Gaussian comparison principles, which are reviewed in Appendix \ref{sec:gaussian-comparison}. After, Section \ref{sec:orth} provides our results on orthogonal matrices, which rely on geometric arguments. Finally, Section \ref{sec:open-directions} highlights open directions in the study of random games that we find particularly relevant to computer science.

\section{Preliminaries} \label{sec:prelim}
\subsection{Notations and definitions}
\textit{Generals.} $M\in \Rbb^{n\times m}$ represents a matrix with $n$ rows and $m$ columns, with $n,m\geq 2$. By symmetry, we will often restrict our attention to the case where $m\geq n$.  We use $I_n\in\Rbb^{n\times n}$ to denote the identity matrix (when $n=m$).   $\Delta_n = \{\bx\in \Rbb^n: \forall i: x_i \geq 0 ~\sum_i x_i =1 \}$ is the $n$-dimensional probability simplex. We will adopt the convention that player $1$ is the \textit{row player}, choosing a pure strategy $i\in [n]$, or a mixed strategy $\bx\in \Delta_n$; and that player $2$ is the \textit{column player}, choosing a pure strategy $j\in [m]$, or a mixed strategy $\by\in \Delta_m$. We denote by $\onebf_n\in \Rbb^n$ the vector of all ones. When the dimension is clear from context, for example, in the context of matrix-vector multiplication, we will drop the subscript. We refer to the vector $\frac{1}{n}\onebf_n \in \Delta_n$ as the uniform strategy.  

\textit{Gaussian variables.} For a real-valued random variable $X$, we denote its expectation by $\Ebb(X)$, its variance by $\Var(X)$, and its standard deviation by $\sigma(X) = \sqrt{\Var(X)}$.  A real-valued random variable $X$ is said to follow a standard Gaussian distribution $X\sim \Ncal(0,1)$ if it has density $\varphi(x) = \frac{1}{\sqrt{2\pi}} \exp\!\left(-\frac{x^2}{2}\right)$. 
Note that $\Ebb(X\oneb_{X>a}) = \int_{a}^\infty x\varphi(x)dx = \varphi(a)$ because $\forall x\in \Rbb: x\varphi(x) =-\varphi'(x)$. We will also use the following classical tail inequality for $x\geq 0 : \Pbb(X\geq x) = 1-\Phi(x) \leq \exp\left(-\frac{x^2}{2}\right)$, where $\Phi(x) = \int_{-\infty}^x\phi(t)dt$ is the cumulative distribution function. 

\textit{Gaussian vector.} For a symmetric positive definite matrix $\Sigma\in S^n_{++}$, we say that a vector $\bg\in \Rbb^n$ follows a centered multivariate Gaussian distribution with covariance matrix $\Sigma\in \Sbb^n_{++}$ and we write $\bg\sim \Ncal(0,\Sigma)$ if it has density $\varphi_{\Sigma}(\bg) = \frac{1}{\sqrt{(2\pi)^n \det\Sigma}}\exp(-\frac{1}{2}\bg^T\Sigma^{-1}\bg)$.  

\textit{Chi distribution.} The norm of the Gaussian vector $\bg\sim \Ncal(0,I_n)$ follows the so-called chi distribution with $n$ degrees of freedom, $||\bg||\sim \chi_n$. Classically, one has, $a_n = \Ebb(||\bg||) = \sqrt{2}\frac{\Gamma(\frac{n+1}{2})}{\Gamma(\frac{n}{2})},$
where $\Gamma(x) = \int_{0}^\infty t^{x-1}\exp(-t)dt$ is the Gamma function. Observing that $(a_n)_{n\in \Nbb}$ satisfies $a_n a_{n+1} = n$ and $a_n^2 \leq \Ebb(||\bg||^2) = n$ (by Jensen), we obtain for all $n\geq 1$,
\begin{equation}\label{eq:bound-a}
    \sqrt{n} -\frac{1}{2\sqrt{n}}\leq a_n \leq \sqrt{n}.
\end{equation}
Note that a complete series expansion of $(a_n)_{n\in \Nbb}$ can be obtained by induction (see, e.g., \cite{gordon1985some}), but Equation \eqref{eq:bound-a} suffices for the purposes of this paper.

\textit{Positive part.} For a Gaussian vector $\bg\sim \Ncal(0,I_n)$, we denote by $\bg^+ = (\max\{0,g_i\})_{i\in[n]}$ its positive component and $\bg^- = (\max\{0,-g_i\})_{i\in [n]}$ its negative component so that $\bg = \bg^+-\bg^-$. For $i\in [n]$, direct computation gives $\Ebb(g_i^+) 
= \frac{1}{\sqrt{2\pi}}$. By Jensen's inequality $\Ebb(||\bg^+||)\leq \sqrt{\Ebb(||\bg^+||^2)} = \sqrt{\frac{n}{2}}$. We also have the lower bound $\Ebb(||\bg^+||)\geq \sqrt{\frac{n}{2}} - \frac{3}{\sqrt{n}}$ (derived in Appendix \ref{ap:technical-binom}).

\textit{Gaussian concentration.} (see, e.g., Th. 5.6 of \cite{boucheron2013}). Let $f:\Rbb^n\rightarrow\Rbb$ denote a $1$-Lipschitz function for the Euclidean norm and $\bg \sim \Ncal(0,I_n)$. Then, for all $t>0$,
$$\Pbb(f(\bg)-\Ebb(f(\bg))>t)\leq \exp\left(-t^2/2\right).$$ Throughout the paper, we will usually apply this concentration inequality to the following $1$-Lipschitz functions $\bg\rightarrow \frac{1}{\sqrt{n}}\sum_i g_i^+$,  as well as to $\bg \rightarrow||\bg||$, and $\bg \rightarrow ||\bg^+||$. 

\subsection{Preliminaries on random games}
In this section, we present some important preliminary results. 
\begin{proposition}[see, e.g., \cite{jonasson2004optimal} for further details]\label{prop: prelim}
Let $M\in\Rbb^{n\times m}$ be a random matrix with i.i.d. standard Gaussian entries (cf. Section \ref{sec:main}), or $M\in\Rbb^{n\times n}$ be a uniformly random orthogonal matrix with $m=n$ (cf. Section \ref{sec:orth}). With probability one, the solution $(\bx,\by,v)$ of Equation \eqref{eq:vonNeumann} is unique and the following equalities hold: 
\begin{align*}
    C &= \{j\in [m] : y_j>0\} = \{j\in [m] : (\bx^TM)_j = v\},\\
    R &= \{i\in [n] : x_i>0\} = \{i\in [n] : (M\by)_i = v\},
\end{align*}
where $C$ the set of supporting columns and $R$ the set of supporting rows have the same cardinality $|C|=|R|$.  We also have, denoting by $(M_{RC}, \bx_R,\by_C)$ the restrictions of $(M,\bx,\by)$ to rows $R$ and columns $C$, that $M_{RC}$ is almost surely invertible and that
\begin{equation}\label{eq:explicit}
    \by_C = v M_{RC}^{-1}\onebf\quad\text{and}\quad  \bx_R = v \onebf^T M_{RC}^{-1} \quad \text{with} \quad v = \frac{1}{\onebf^T M_{RC}^{-1}\onebf}.
\end{equation}
\end{proposition}

\begin{proof}[Proof sketch]
    Solutions $(\bx,\by,v)$ of Equation \eqref{eq:vonNeumann}  are given by the following linear optimization problem.
\begin{equation}\label{eq:primal-dual}
\begin{aligned}
&\text{Primal}\\
\text{minimize} &\quad v \\
 \text{variables} &\quad v \in \mathbb{R}, \quad \bx \in \mathbb{R}_+^n \\
 \text{subject to} &\quad \bx^T \onebf = 1 \\
 \phantom{\text{subject to}} &\quad \bx^T M \leq v \onebf
\end{aligned}
\qquad\qquad
\begin{aligned}
&\text{Dual}\\
\quad
 \text{maximize} &\quad v \\
 \text{variables} &\quad v \in \mathbb{R}, \quad \by \in \mathbb{R}_+^m \\
 \text{subject to} &\quad \onebf^T \by = 1 \\
 \phantom{\text{subject to}} &\quad M \by \geq v \onebf
\end{aligned}
\end{equation}
Assuming that no edge of the constraint polytope is orthogonal to the objective vector, we get the uniqueness of $(\bx,\by,v)$. This condition is satisfied almost surely for random vectors with a density.

In the linear program \eqref{eq:primal-dual}, $\by \in \Rbb_+^m$ are the dual variables corresponding to the constraint $\bx^T M \leq v\onebf$, and  $\bx\in \Rbb_+^n$ are the dual variables corresponding to the constraints $M\by\geq v\onebf$. Therefore, considering the optimal primal dual pair $(\bx,v)$, $(\by,v)$ and denoting 
\begin{equation*}
    \begin{cases}
    C = \{j\in [m]: (\bx^TM)_j = v\}\\
    C' = \{j\in [m] : y_j=0\}
    \end{cases} \quad \text{and} \quad 
    \begin{cases}
        R = \{i\in [n]: (M\by)_i = v\}\\
        R' = \{i\in [n] : x_i = 0\}
    \end{cases},
\end{equation*}
we have by complementary slackness that $C\cup C' = [m]$ and $R\cup R' =[n]$. Furthermore, assuming that the constraints are nondegenerate (i.e., the number of active constraints at any feasible point is at most equal to the dimension), the optimal solution $(\by,v)$ saturates exactly $m$ inequality constraints, and $(\bx,v)$ saturates exactly $n$ inequality constraints. Therefore, $|C| + |R'| = n$ and $|R| + |C'| = m$. From there, we deduce that $C = [m]\setminus C'$ and $R = [n]\setminus R'$. Note that non-degeneracy of the constraint polytope is satisfied almost surely for the random matrices we consider. Finally, \eqref{eq:explicit} follows directly from the definition of $C$ and $R$.
\end{proof}
\begin{remark}[Symmetry of game value]\label{rem:symmetry} It directly follows from the linear programming formulation that for any matrix $M\in\Rbb^{n\times m}$ we have $v(-M^T) = -v(M)$. Thus, for random square matrices, if we have the following equality in distribution, $M=^{dist}-M^T$, we immediately get $\Ebb(v(M)) = 0$.
\end{remark}
We now present a simple tail bound on the value of a random two-player game. This result follows directly from the observation that the row player can always use the uniform strategy. 
\begin{proposition}\label{prop:naive-uniform}
    Let $M\in \Rbb^{n\times m}$ be a random matrix whose entries are i.i.d. $\Ncal(0,1)$, then, \begin{equation}\label{eq:tail-bound-naive}
        \Pbb(v(M)\geq t)\leq m\exp(-nt^2/2).
    \end{equation}
\end{proposition}
\begin{proof}
    Since the uniform strategy $\frac{1}{n}\onebf\in \Delta_n$ can be used by the row player, we clearly have that
    $v(M)\leq \frac{1}{n}\max_{j\in [m]}\left(\onebf^TM\right)_j$, which is the maximum of $m$ independent and identically distributed variables following $\Ncal(0,1/n)$.
    Therefore, we have, by union bound, for all $t\geq 0$,
    \begin{align*}
        \Pbb(v(M)\geq t)&\leq m(1-\Phi(t\sqrt{n}))\\
        &\leq m\exp(-nt^2/2),
    \end{align*}
    where the last line uses the standard Gaussian tail inequality $1-\Phi(x)\leq \exp(-x^2/2)$.
\end{proof}
\begin{remark}[Interpretation and consequences]
An immediate consequence of this result is the following high-probability bound: for all $t\geq 0$, and $M\in \Rbb^{n\times m}$ with i.i.d. Gaussian entries,
\begin{equation*}
    \Pbb\left(-t\sqrt{\frac{\log n}{m}}\leq v(M)\leq t\sqrt{\frac{\log m}{n}}\right)\geq 1-2\exp(-t^2/2).
\end{equation*}
This bound is well-known in the literature (see, e.g., \cite{faris1987value}), but it is imprecise. Firstly, for $n=m$, it does not match the $\Theta(1/n)$ conjectured order for $v(M)$. Secondly, for $m>>n+\sqrt{n}$, it does not reflect the fact (shown in Cover \cite{cover1966probability}) that the value of the game is almost always positive.   
\end{remark}
We conclude the section with the following lower bound on the variance of $v(M)$. 
\begin{proposition}\label{prop:lb}Let $M\in \Rbb^{n\times m}$ be a random matrix whose entries are i.i.d. $\Ncal(0,1)$, then, 
\begin{equation*}
    \Var(v(M))\geq \frac{1}{nm}.
\end{equation*}
\end{proposition}
\begin{proof}
Consider $M' = M-a\Jbb$, for $a = \frac{1}{nm}\sum_{ij} M_{ij}$ and $\Jbb\in \Rbb^{n\times m}$ the all ones matrix. Note that $v(M) =v(M')+a$ and that $M'$ and $a\Jbb$ are independent as they are the projections of the Gaussian vector $M$ on orthogonal spaces. This gives 
$$\Var(v(M)) = \Var(v(M'))+\Var(a) \geq \Var(a)=  \frac{1}{nm}.$$
\end{proof}
Recall that by Remark \ref{rem:symmetry}, for Gaussian matrices with $n=m$, we have $\Ebb(v(M))=0$ by symmetry. Thus, Proposition \ref{prop:lb} supports the experimental conjecture of \cite{faris1987value} that typical values of $v(M)$ for square matrices are of order $1/n$. In further support of this conjecture, we refer the reader to Figure \ref{fig:scaling} of Appendix \ref{ap:experiment}, where we provide our own experimental evidence.

\section{Random games on Gaussian matrices}\label{sec:main}
\subsection{Main result and implications}
The goal of this section is to prove the following statement, which is our main technical result for Gaussian matrices.
\begin{theorem}\label{th:main}
We fix $n\leq m\leq 2n$. The value of the random game defined by 
\begin{equation*}
    v(M) = \min_{\bx\in \Delta_n}\max_{\by\in\Delta_m} \bx^T M \by,
\end{equation*}
for $M\in \Rbb^{n\times m}$ a matrix with i.i.d. $\Ncal(0,1)$ entries, satisfies, for all $t>0$,
\begin{equation*}
    \frac{1}{4}\frac{\sqrt{m}-\sqrt n}{n}-\frac{t}{n}\leq v(M)\leq 40\sqrt{2}\frac{\sqrt{m}-\sqrt{n}}{n}+\frac{t}{n},
\end{equation*}
with probability at least $1-c_1\exp(-c_2 t^2)-c_3\exp(-c_4 n)$, for some universal constants $c_1,c_2,c_3,c_4>0$.
\end{theorem}
Theorem \ref{th:main} has the following two immediate consequences: 
\begin{enumerate}[label=\roman*.]
    \item For square matrices $(m=n)$, it implies that $\sigma(v(M))= \Ocal(1/n)$, and thus, that $\sigma(v(M)) = \Theta(1/n)$ by the lower bound of Proposition \ref{prop:lb}. Indeed, one can write, using the tail bound of Theorem \ref{th:main} for the smaller deviations of $v(M)$ together with the tail bound of Equation \eqref{eq:tail-bound-naive} for larger deviations of $v(M)$,
    \begin{align*}
        \sigma(v(M))^2 = \Ebb(v(M)^2)&= \int_{0}^\infty \Pbb(v(M)^2\geq s)ds \\
        &= 
    2\int_{0}^1 \Pbb(v(M)\geq \sqrt s)ds+
    2\int_{1}^\infty \Pbb(v(M)\geq \sqrt s)ds
    \\
    &\leq 2\int_0^1 c_1\exp(-c_2n^2s)+c_3\exp(-c_4n)ds + 2\int_1^\infty n\exp(-ns/2)ds\\
    &\leq 2\frac{c_1}{c_2n^2} + 2c_3\exp(-c_4n)+ 2\exp(-n/2)\\
    &=\Ocal\left(\frac{1}{n^2}\right).
    \end{align*}
    \item For rectangular matrices $(n<m\leq 2n)$, by a similar argument, Theorem \ref{th:main} provides the order of the expectation of $v(M)$, 
    \begin{equation*}
        \Omega\left(\frac{\sqrt{m}-\sqrt{n}}{n}\right)-\Ocal\left(\frac{1}{n}\right) \leq \Ebb(v(M))\leq \Ocal\left(\frac{\sqrt{m}-\sqrt{n}}{n}\right)+\Ocal\left(\frac{1}{n}\right),
    \end{equation*}
for which full details are provided in Appendix \ref{ap:expectation-rectangular}, along with an explanation about our (slightly) unconventional use of Landau notations in the above equation. Note that in the regime $m = n +\lambda\sqrt{n}$, for $\lambda>0$, we have that $\frac{\sqrt{m}-\sqrt{n}}{n}$ scales as $\frac{\lambda}{n}$, therefore, we capture an aspect ratio effect even for `slightly' rectangular matrices: in that regime, the aspect ratio competes with the noise to determine the quantity $v(M)$. 
\end{enumerate}

\begin{remark}[Regime where $m>2n$ or $n>m$] Theorem \ref{th:main} also captures tail bounds for the regime where $m>2n$ up to a $\sqrt{\log m}$ factor. Specifically, in that regime, restricting the column player to play on the left submatrix of shape $(n,2n)$ can only reduce the value of $v(M)$, and thus Theorem~\ref{th:main} provides a $\Omega\left(\frac{1}{\sqrt{n}}\right)$ lower bound with high probability. Also, in that regime, Proposition~\ref{prop-norm} gives a nearly matching upper bound, in  $\Ocal\left(\sqrt{\log m}/\sqrt{n}\right)$. Finally, note that regimes where $n>m$ follow immediately by symmetry. 
\end{remark}

\subsection{An approach via comparison principles}\label{sec:approach-comp}
Gaussian comparison inequalities are fundamental tools in probability theory \cite{ledoux2013probability}. They have found applications in computer science, including streaming algorithms \cite{braverman2016beating}, dimensionality reduction \cite{dirksen2016dimensionality}, compressed sensing \cite{plan2012robust}, and machine learning \cite{miolane2021distribution,thrampoulidis2015regularized}. 

A classic example of a comparison inequality is Slepian's lemma \cite{slepian1962one}, stating that the expectation of the maximum of $n$ positively correlated standard Gaussians is smaller than the expectation of the maximum of $n$ independent standard Gaussians. More generally, Gaussian comparison inequalities allow us to compare properties of a complex (e.g., correlated) Gaussian process to properties of a simple (e.g., uncorrelated) Gaussian process. Note that this idea is very similar to the so-called Poisson approximation, in ``balls in bins'' settings (see, e.g., its classical application to hashing in \cite{mitzenmacher2017probability}). Slepian's lemma is a direct implication of Kahane's inequality, which is reviewed in Appendix~\ref{sec:gaussian-comparison} (Theorem~\ref{th:kahane}) of this paper, along with more advanced Gaussian comparison inequalities. 

In this section, we use a specific Gaussian comparison principle, known as Gordon's inequality \cite{gordon1985some}. There are many equivalent formulations of Gordon's inequality, and we state below, in Theorem \ref{th:gordon}, the version most directly relevant to our purposes. We provide background on this inequality and alternative statements in Appendix \ref{sec:gaussian-comparison}, as well as a sketch of proof from first principles, for completeness. 

Theorem \ref{th:gordon} allows to reduce the study of $v(M)$ to the study of another `simpler' stochastic quantity $\Phi_2(\bg,\bh)$. This is the quantity we analyze in the proof of Theorem~\ref{th:main}. While this approach is more direct than the geometric approach of Section \ref{sec:orth} for the study of orthogonal games, we believe that the geometric approach is also very instructive and we refer the reader to the outline in Section~\ref{sec:geom-approach}.
\begin{restatable}[Gordon's inequality]{theorem}{thgordon} \label{th:gordon}
Suppose $D_x\subset \Rbb^n$ and $D_y\subset \Rbb^m$ are compact sets. Let $G\in \Rbb^{n\times m}$ be a matrix with i.i.d. $\Ncal(0,1)$ entries, and $\bg\sim \Ncal(0,I_n)$ and $\bh\sim\Ncal(0,I_m)$. Define,
\begin{align*}
    \Phi_1(G) &= \min_{\bx\in D_x}\max_{\by\in D_y} \bx^TG\by,\\
    \Phi_2(\bg,\bh) &= \min_{\bx\in D_x}\max_{\by\in D_y} ||\by||\bg^T\bx + ||\bx||\bh^T\by,
\end{align*}
then we have, for any $t\in \Rbb$,
\begin{equation}\label{eq:base-eq}
    \Pbb(\Phi_1(G)\le t)\leq 2\Pbb(\Phi_2(\bg,\bh)\le t).
\end{equation}
\end{restatable}

\subsection{Proof of Theorem \ref{th:main}}
We now lay down the proof of Theorem \ref{th:main}. We apply Gordon's principle (Theorem \ref{th:gordon}) to
    $$v(M) = \min_{\bx\in \Delta_n}\max_{\by\in\Delta_m} \bx^T M \by,$$
which entails, for all $t\in \Rbb$,
    \begin{align}\label{eq:gordon-app1}
        \Pbb(v(M)\leq t)\leq 2\Pbb(\Phi_2(\bg,\bh)\leq t),
    \end{align}
    where $\bg\sim \Ncal(0,I_n)$ and $\bh\sim \Ncal(0,I_m)$ are sampled independently and where
\begin{equation}\label{eq:lp-init}
   \Phi_2(\bg,\bh) = \min_{\bu\in \Delta_n}\max_{\bv\in \Delta_m}  ||\bv||\bg^T\bu+||\bu||\bh^T\bv \tag{P}.
\end{equation}
Using the symmetry $v(M) = -v(-M^T)$ from Remark \ref{rem:symmetry}, and exploiting properties of $\Phi_2(\cdot,\cdot)$ (see Appendix \ref{ap:symmetry} for details), we also obtain, for all $t\in \Rbb$,
\begin{equation}
    \Pbb(v(M)\geq t) \leq 2\Pbb(\Phi_2(\bg,\bh)\geq t).
\end{equation}

We now study the quantity $\Phi_2(\bg,\bh)$. We fix two vectors $\bh$ and $\bg$, we assume that $\bg^-\neq 0$ and $\bh^+\neq 0$. We obtain the following bounds on \eqref{eq:lp-init} 
by first considering  $\bv = \frac{\bh^+}{\onebf^T\bh^+}$, and then considering $\bu = \frac{\bg^-}{\onebf^T\bg^-}$,
\begin{equation}\label{eq:upper-lower-bounds}
    \frac{||\bh^+||}{\onebf^T\bh^+}\min_{\bu\in \Delta_n}\left\{\bg^T\bu+||\bh^+||||\bu||\right\}\leq \Phi_2(\bg,\bh)\leq \frac{||\bg^-||}{\onebf^T\bg^-}\max_{\bv\in \Delta_m}\left\{\bh^T\bv-||\bg^-||||\bv||\right\}.
\end{equation}
We thus focus on the following optimization problem, for some fixed vector $\bh\in\Rbb^m$, and some fixed parameter $\gamma>0$,
\begin{equation}\label{eq:main-LP}
    \max_{\bv\in \Delta_m}f(\bv) \quad \text{where}\quad f(\bv) = \bh^T\bv -\gamma||\bv||.\tag{R}
\end{equation}
The function $:\bv\rightarrow f(\bv)$ is strongly concave and differentiable on $\Delta_m$, which is a closed and bounded convex set. We write the Lagrangian of \eqref{eq:main-LP}, introducing dual variables $\mu\in\Rbb$ and $\lambda_j\geq 0$,
$$\Lcal(\bv,\mu,(\lambda_j)) = \bh^T\bv - \gamma ||\bv||+\sum_{j\in [m]}\lambda_j v_j+\mu\left(1 - \sum_{j\in [m]} v_j\right)= \sum_{j\in [m]} v_j(h_j+\lambda_j-\mu)-\gamma||\bv||+\mu.$$
By strong duality and strong concavity, we have the existence of a primal dual pair $(\bv,\blambda,\mu)$ satisfying primal and dual feasibility, as well as the first order condition $v_j = \frac{||\bv||}{\gamma}(h_j+\lambda_j -\mu)$ and complementary slackness $\forall j\in [m] : v_j \lambda_j = 0$. In particular, $v_j>0$ entails $v_j = \frac{||\bv||}{\gamma}(h_j-\mu)$ and $v_j = 0$ entails $h_j-\mu = -\lambda_j\leq 0$. Therefore, we can write for all $j\in [m]$,
\begin{equation*}
    v_j = \frac{||\bv||}{\gamma}(h_j-\mu)^+.
\end{equation*}
This imposes the constraint
\begin{equation}
||(\bh-\mu\onebf)^+||= \gamma    \label{eq:norm-h}
\end{equation}
which defines the value of $\mu\in \Rbb$ uniquely because $:\mu\rightarrow ||(\bh-\mu\onebf)^+||$ is strictly decreasing from $+\infty$ until it hits $0$. Also, using the fact that $\bv\in \Delta_m$, we get that
\begin{equation*}
    \forall i \in[m]: v_j = \frac{1}{\sum_{j'\in [m]} (h_{j'}-\mu)^+ }(h_j-\mu)^+.
\end{equation*}
Overall, we obtain the following value for \eqref{eq:main-LP}, 
\begin{equation*}
    \eqref{eq:main-LP} =  \frac{1}{\sum_{j\in [m]} (h_j-\mu)^+} \left(\bh^T(\bh-\mu\onebf)^+ -\gamma^2\right), 
\end{equation*}
which, noting that $\bh^T(\bh-\mu\onebf)^+\leq \bh^+(\bh-\mu\onebf)^+$, applying Cauchy-Schwarz, and using \eqref{eq:norm-h}, entails, 
\begin{equation}
      \eqref{eq:main-LP}
    \leq  \frac{\gamma}{\sum_{j\in[m]} (h_j-\mu)^+} \left(||\bh^+|| - \gamma\right).
\end{equation}
Now going back to the right-hand side of Equation \eqref{eq:upper-lower-bounds}, and taking $\gamma = ||\bg^-||$ in \eqref{eq:main-LP}, we have the following bound,
\begin{equation}\label{eq:ub-phi}
    \Phi_2(\bg,\bh) \leq \frac{||\bg^-||}{\onebf^T\bg^-}\frac{||\bg^-||}{C(\mu) }(||\bh^+||-||\bg^-||),
\end{equation}
where $C(\mu) = \sum_{j\in[m]} (h_j-\mu)^+$ for $\mu\in\Rbb$ the constant satisfying $||(\bh-\mu\onebf)^+|| = ||\bg^-||$. 

Using the left-hand side of Equation \eqref{eq:upper-lower-bounds}, and also noticing that $\min_{\bu\in \Delta_n}\left\{\bg^T\bu+||\bh^+||||\bu||\right\} = -\max_{\bu\in \Delta_n}\left\{(-\bg)^T\bu-||\bh^+||||\bu||\right\}$, we obtain,
\begin{equation}\label{eq:lb-phi}
    \frac{||\bh^+||}{\onebf^T\bh^+}\frac{||\bh^+||}{C'(\mu')}(||\bh^+||-||\bg^-||) \leq \Phi_2(\bg,\bh),
\end{equation}
where $C'(\mu') = \sum_{i\in[n]} (-g_i-\mu')^+$ for $\mu'\in \Rbb$ the constant satisfying $||(-\bg-\mu'\onebf)^+|| = ||\bh^+||$.

The following simple concentration results are shown in Appendix \ref{sec:concentration}. For some constant $c_4>0$,
\begin{equation*}
    \begin{cases}
    \Pbb(\bg^-\neq 0 ~~ \text{and} ~~ \bh^+\neq 0) \geq 1-\Ocal(\exp(-c_4 n)),\\
\Pbb(||\bg^-||\leq \sqrt{n})\geq 1-\Ocal(\exp(-c_4 n)),\\
\Pbb(\onebf^T\bg^-\geq n/4)\geq 1-\Ocal(\exp(-c_4 n)),\\
\Pbb(C(\mu)\geq m/20)\geq 1-\Ocal(\exp(-c_4 n)),\\
\Pbb(||\bh^+||\geq \sqrt{m}/2)\geq 1- \Ocal(\exp(-c_4 n)),\\
\Pbb(\onebf^T\bh^+\leq m/2)\geq 1-\Ocal(\exp(-c_4 n)),\\
\Pbb(C'(\mu')\leq \sqrt{2}~n)\geq 1-\Ocal(\exp(-c_4 n)),
    \end{cases}
\end{equation*}
and, for all $t>0$,
\begin{equation*}
    \begin{cases}
      \Pbb(\sqrt{m/2}-3/\sqrt{m}-t \leq ||\bh^+||\leq \sqrt{m/2} + t)\geq 1-2\exp(-t^2/2),\\
\Pbb(\sqrt{n/2}-3/\sqrt{n}-t\leq ||\bg^-||\leq \sqrt{n/2}+t)\geq 1-2\exp(-t^2/2).  
    \end{cases}
\end{equation*}
Thus, using Equation \eqref{eq:ub-phi}, we obtain by union bound  
\begin{equation*}    
\Pbb\left(\Phi_2(\bg,\bh) \leq \frac{80}{m}\left(\sqrt{m/2}-\sqrt{n/2}+3/\sqrt{n}+2t\right) \right)\geq 1 - \Ocal(\exp(-c_4 n))-2\exp(-t^2/2),
\end{equation*}
and using \eqref{eq:lb-phi}, we obtain,
\begin{equation*}
\Pbb\left(\Phi_2(\bg,\bh)\geq \frac{1}{2\sqrt 2 n}
\left(\sqrt{m/2}-\sqrt{n/2}-3/\sqrt{m}-2t\right)
\right)\geq 1- \Ocal(\exp(-c_4 n))-2\exp(-t^2/2).
\end{equation*}
This implies the theorem for the right choice of constants $c_1,c_2,c_3,c_4>0$.

\section{Random games on orthogonal matrices}\label{sec:orth}

\subsection{A geometric approach}\label{sec:geom-approach}
We now consider the case where the game matrix is drawn uniformly at random from the set of square orthogonal matrices $O_n = \{Q\in \Rbb^{n\times n} : Q^TQ = I_n\}$. We keep the discussion informal in this paragraph, to convey the spirit of our approach, and defer formal definitions to Section~\ref{sec:background-convex}. Our techniques exploit geometric results on convex cones.

We begin by providing an initial intuition for the connection between random games and convex geometry. We denote by $K = \{\by\in \Rbb^n : \by\geq 0\}$ the positive orthant, and by $Q^TK = \{Q^T\by\in \Rbb^n : \by\geq 0\}$ the orthant rotated by the random orthogonal matrix $Q^T$. We observe that $v(Q)\geq 0$ if and only if there exists $\by\in \Rbb^n\setminus\{0\}$ such that $\by \geq 0$ and $Q\by\geq 0$.  This rewrites as $v(Q)\geq 0 \Leftrightarrow K\cap Q^TK \neq \{0\}$. Thus, the probability that the value of the game defined by $Q$ is positive, $\Pbb(v(Q)\geq 0)$, which equals $1/2$ by symmetry, is also equal to the probability that a randomly rotated orthant of $\Rbb^n$ has non-trivial intersection with the positive orthant. To the best of our knowledge, this observation is also the most direct proof of the counterintuitive fact that two randomly rotated orthants intersect with probability $1/2$, independently of the ambient dimension $n$. 

This discussion hints at how Cover \cite{cover1966probability} was able to obtain an explicit expression for $\Pbb(v(Q)\geq 0)$, for rectangular matrices. Specifically, his proof relies on a formula due to Schläfli \cite{schlafli1950gesammelte} for the probability that a random subspace of dimension $n<m$ intersects the positive orthant of $\Rbb^m$. However, the argument falls short when it comes to providing bounds on the typical values of $v(Q)$, and not just its sign. To this end, we aim to bound the probability that the intersection of randomly rotated orthants is sufficiently `wide'. Thankfully, geometric results of this flavor were developed over the years: specifically, the so-called \textit{conic kinematic formula} provides an exact formula for the probability that two randomly rotated convex cones intersect, in terms of their intrinsic volumes (see the book of \cite{schneider2008stochastic}). More recently, a practical approximate variant of the conic kinematic formula was proposed by \cite{amelunxen2014living} in the context of compressed sensing.

Because the random orthogonal model is quite involved and not as common as the random Gaussian model, we restrict our study to the case of square matrices $n=m$. 
Specifically, our main result is the following.
\begin{theorem}\label{th:main2}
We fix $n\geq 2$. The value of the random game defined by 
\begin{equation*}
    v(Q) = \min_{\bx\in \Delta_n}\max_{\by\in\Delta_n} \bx^T Q \by,
\end{equation*}
where $Q\in \Rbb^{n\times n}$ is drawn from the Haar measure on orthogonal matrices, satisfies, for all $t>0$, 
\begin{equation}
    v(Q)\leq \frac{t}{n\sqrt{n}},\label{eq:bound-orth}
\end{equation}
with probability at least $1-c_1\exp(-c_2 t^2)-c_3\exp(-c_4 n)$, for universal constants $c_1,c_2,c_3,c_4>0$.
\end{theorem}

\begin{remark}[Interpretation and comparison with Theorem \ref{th:main}.] Theorem \ref{th:main2} echoes the scaling of Theorem \ref{th:main} for Gaussian matrices. The extra $1/\sqrt{n}$ factor in \eqref{eq:bound-orth} is natural because rows of Gaussian matrices would be of norm $\sqrt{n}$ with high probability, whereas rows of $Q$ are all of norm $1$. Also, one immediately notices that the uniform strategy for the row player $\bx = \frac{1}{n}\onebf$ induces a value $\frac{1}{n}\max_{j\in [n]}(\onebf^TQ)_j$ of order $\sqrt{\log n}/n$, because, by invariance by rotation of the Haar measure, the vector $\onebf^TQ$ is distributed uniformly in the Euclidean sphere of radius $n$, and thus its largest coordinate is of order $\sqrt{\log n}$. In summary, like for Gaussian matrices, Theorem \ref{th:main2} shows that the optimal strategy has an `advantage of rationality' in the order of $\Theta(\sqrt{\log n}/\sqrt{n})$ over the uniform strategy.     
\end{remark}

\subsection{Background on convex geometry}\label{sec:background-convex} 
\textit{Orthonormal matrices.} The set of orthogonal matrices $O_n =\{Q\in \Rbb^{n\times n} :Q^TQ = I_n\}$ is stable by multiplication and inverse, and thus forms a subgroup of the invertible matrices. The Haar measure is the uniform measure over orthogonal matrices, and is further denoted by $\Ucal(O_n)$. Formally, it is defined as the unique probability measure that is stable by multiplication with elements of the group, i.e., for any $Q'\in O_n$ we have $Q\sim \Ucal(O_n) \implies Q'Q\sim \Ucal(O_n)$ and $QQ'\sim \Ucal(O_n)$. A simple constructive way to draw a matrix from the Haar measure is to perform the singular value decomposition of a matrix $M\in \Rbb^{n\times n}$ with i.i.d. Gaussian entries. 

\textit{Convex cones.} A cone $K\subset\Rbb^n$ is a set that satisfies for all $\lambda\geq 0$ that $x\in K \implies \lambda x\in K$. A convex cone is a cone that is convex in $\Rbb^n$. For an orthogonal matrix $Q\sim \Ucal(O_n)$ and a convex cone $K$, we denote by $QK = \{Q\bx : \bx\in K\}$ the cone `rotated' by matrix $Q$.\footnote{This is a slight abuse of language.} The rotation of a convex cone remains a convex cone. 

\textit{Conic kinematic formula.} The probability that two randomly rotated cones intersect can be expressed in terms of their intrinsic volumes. For our purposes, an approximate version of the conic kinematic formula, which avoids the definition of intrinsic volumes, suffices. It relies on the following definition.
\begin{definition}[Statistical dimension \cite{amelunxen2014living}] The statistical dimension $\delta(K)$ of a convex cone $K\subset\Rbb^n$ is given by 
\begin{equation*}
    \delta(K) = \Ebb(||\Pi_K(\bg)||^2),
\end{equation*}
where $\bg\sim \Ncal(0,I_n)$, $||\cdot||$ is the Euclidean norm, and $\Pi_K(\cdot)$ denotes the Euclidean projection onto the cone $K$, $\forall \bx\in \Rbb^n : \Pi_K(\bx) = \argmin_{\by\in K}||\bx-\by||$.
\end{definition}
For example, the statistical dimension of the positive orthant $K = (\Rbb_+)^n$ satisfies $\delta(K) =\Ebb(||\bg^+||^2) = n/2$. We can now state the approximate kinematic formula.
\begin{theorem}[Approximate kinematic formula \cite{amelunxen2014living}]\label{th:approx-kin-form} Let $\eta\in (0,1)$. Let $K_1$ and $K_2$ be convex cones in $\Rbb^n$ and $Q\sim \Ucal(O_n)$, then, denoting $a_\eta = \sqrt{8\log(4/\eta)}$, for $\eta\in (0,1)$,
\begin{align*}
    \delta(K_1)+\delta(K_2)\leq n-a_\eta\sqrt{n} \quad &\implies \quad \Pbb(K_1\cap QK_2\neq\{0\})\leq \eta,\\
    \delta(K_1)+\delta(K_2)\geq n+a_{\eta}\sqrt{n}\quad &\implies\quad \Pbb(K_1\cap QK_2\neq \{0\})\geq 1-\eta.
\end{align*}
\end{theorem}
\begin{remark}[Gaussian width] The statistical dimension of a convex cone is a good approximation of another concurrent notion called the Gaussian width, which preexisted statistical dimension. Gaussian width is defined by $\Ebb(\sup_{\by\in K: ~||\by||=1}\by^T\bg)$. For our purposes, we could work with the two quantities arbitrarily, and we opted for the statistical dimension. 
\end{remark}
\subsection{Proof structure of Theorem \ref{th:main2}}
In this section, we provide a proof of Theorem \ref{th:main2}. We introduce two key propositions along the way, Proposition \ref{prop:stat-dimension2} and Proposition \ref{prop-norm}, which provide insightful structural properties on the optimal strategy of a random game. They are proven in subsequent sections of the paper. Throughout the proof, we shall use the following family of convex cones, parametrized by $\epsilon\geq 0$,
\begin{align*}
    K(\epsilon) &= \{\bz\in \Rbb^n ~s.t.~ \forall i: z_i\geq \epsilon||\bz||\}.
\end{align*}
We start with the following simple lemma.
\begin{lemma} $(K(\epsilon))_{\epsilon\geq 0}$ is a family of convex cones decreasing for inclusion. $K(\epsilon)$ equals the positive orthant for $\epsilon = 0$, a half-line for $\epsilon =\frac{1}{\sqrt{n}}$, and is reduced to $\{0\}$ for $\epsilon> \frac{1}{\sqrt{n}}$.
\end{lemma}
\begin{proof}
Basic properties of $K(\cdot)$ immediately follow from the properties of the norm.  Note that $\bz\in K(\epsilon)$ implies $||\bz||^2 = \sum_i z_i^2 \geq n\epsilon^2||\bz||^2$, which entails, for $\epsilon = 1/\sqrt{n}$, $K(\epsilon) = \{\lambda\onebf : \lambda\in \Rbb_+\}$.
\end{proof}
We now study the statistical dimension of $K(\epsilon)$, which is defined by  
\begin{align*}
\delta(\epsilon)&=\Ebb(||\Pi_{K(\epsilon)}(\bg)||^2),
\end{align*}
where $\bg\sim \Ncal(0,I_n)$ and $\Pi_{K(\epsilon)}$ is the orthogonal projection on $K(\epsilon)$. In Section \ref{sec:proofpropdim}, we provide the following bound on the statistical dimension of $K(\epsilon)$.
\begin{restatable}{proposition}{propdim}\label{prop:stat-dimension2}
For $\epsilon\geq 0$, the statistical dimension of $K(\epsilon)$ satisfies $\delta(\epsilon)\leq \frac{1}{2}n-\frac{1}{8}\epsilon n\sqrt{n}+ 2 \epsilon^2 n^2$.
\end{restatable}
We then introduce the quantity 
\begin{equation*}
    v'(Q) = \max_{\substack{
    \by\geq 0\\
    ||\by||_2 = 1}} \min_{i\in[n]} (Q\by)_i,
\end{equation*}
and observe the following equivalence, for any $\epsilon>0$,
\begin{align}\label{eq:equivalence}
    (v'(Q)\geq \epsilon)&\Leftrightarrow  (K(0)\cap Q^TK(\epsilon) \neq \{0\}),
\end{align}
We now recall that $\delta(0) = n/2$ because $K(0)$ is the positive orthant, and that Proposition \ref{prop:stat-dimension2} provides an upper bound on the value of $\delta(\epsilon)$. Using $\epsilon = t/n$ for $t\leq \sqrt{n}/16$ in this upper bound, we get that
\begin{align*}
\delta(0)+\delta(t/n)\leq n -t\sqrt{n}/16.
\end{align*}
We can thus apply Theorem \ref{th:approx-kin-form} and the equivalence established in Equation \eqref{eq:equivalence} to obtain
\begin{equation}\label{eq:vp}
    \Pbb\left(v'(Q)\geq \frac{t}{n}\right) \leq 4\exp(-t^2/32).
\end{equation}
We conclude by showing that bounding $v'(Q)$ is sufficient to bound $v(Q)$. Specifically, we show that for some universal constants $c_1,c_2>0$,
\begin{equation}\label{eq:v}
    \Pbb\left(v(Q)\leq \frac{c_1}{\sqrt{n}}v'(Q)\right)\geq 1-\exp(-c_2n).
\end{equation}
For this, it suffices to show that the optimal strategy for the column player, defined by $  \by \in \argmax_{\by\in \Delta_n} \min_{i\in [n]}(Q\by)_i$
satisfies $||\by||_2\leq \frac{c_1}{\sqrt{n}}$ 
with high probability. This is precisely the result of the proposition below, which we prove in Section \ref{prop-norm}.
\begin{restatable}{proposition}{propnorm}\label{prop-norm}
The optimal strategy of the column player on $Q\sim\Ucal(O_n)$, defined by
\begin{equation*}
    \by\in \argmax_{\by\in \Delta_n} \min_{i\in[n]} (Q\by)_i,
\end{equation*}
is almost surely unique and satisfies $||\by||\leq \frac{c_1}{\sqrt{n}}$ with probability at least $1-\exp(-c_2n)$, for some universal constants $c_1,c_2>0$. 
\end{restatable}
This finishes the sketch of proof, as the union bound over \eqref{eq:v} and \eqref{eq:vp} entails, for $t\leq \sqrt{n}/16$,
\begin{equation*}
    \Pbb\left(v(Q)\leq \frac{c_1}{\sqrt{n}}\frac{t}{n}\right)\geq 1-\exp(-c_2 n)-4\exp(-t^2/32).
\end{equation*}
We recover the original statement of Theorem \ref{th:main2} after adjusting the constants $c_1,c_2,c_3,c_4>0$.\footnote{Note that we do not have to worry about the regime where $t> cst\cdot \sqrt{n}$ because in that regime, we can pick $c_4>0$ so that the term in $c_1\exp(-c_2 t^2)$ is dominated by the term in $c_3\exp(-c_4 n).$} The rest of our technical effort is spent proving Proposition \ref{prop:stat-dimension2} and \ref{prop-norm}.
\subsection{Proof of Proposition \ref{prop:stat-dimension2}}\label{sec:proofpropdim}
In this section, we prove Proposition \ref{prop:stat-dimension2}, which we restate below.
\propdim*

\begin{proof}
For all $\bg\in \Rbb^n$, we use the orthogonality of $\bg - \Pi_{K(\epsilon)}(\bg)$ and $\Pi_{K(\epsilon)}(\bg)$, which follows from the fact that $K(\epsilon)$ is a convex cone, to write,
\begin{equation*}
    ||\Pi_{K(\epsilon)}(\bg)||^2 = ||\bg||^2-\min_{\bz\in K(\epsilon)}||\bg-\bz||^2.
\end{equation*}
We write the Lagrangian of the optimization problem $\min_{\bz\in K(\epsilon)}||\bg-\bz||^2$ with dual variables $\blambda\in \Rbb_+^n$ as follows,
\begin{equation}\label{eq:lagrangian}
    \Lcal(\bz,\blambda) = ||\bg-\bz||^2 + \sum_{i\in[n]} \lambda_{i}(\epsilon||\bz||-z_i).
\end{equation}
We have, by strong duality,
\begin{equation}\label{eq:weak-duality}
    \min_{\bz\in K(\epsilon)}||\bg-\bz||^2 = \max_{\blambda \in\Rbb_+^n}\min_{\bz\in \Rbb^n}\Lcal(\bz,\blambda) = \min_{\bz\in \Rbb^n}\max_{\blambda\in \Rbb_+^n}\Lcal(\bz,\blambda). 
\end{equation}
We propose to take $\blambda = 2\bg^- \in \Rbb_+^n$ in the right-hand side of \eqref{eq:weak-duality}. Specifically, we use the bound
\begin{equation*}
    \min_{\bz\in K(\epsilon)}||\bg-\bz||^2 \geq \min_{\bz\in \Rbb^n} \Lcal(\bz, 2\bg^-).
\end{equation*}
We shall later show that the minimizer $\bz$ of $\Lcal(\bz,2\bg^-)$ is given by 
\begin{equation}\label{eq:minimizer-explicit}
    \begin{cases}
    \bz = 0 \quad \text{if} \quad ||\bg^+||\leq\epsilon \sum_i g_i^-,  \\
    \bz = \left(1-\frac{\epsilon\sum_{i}g_i^-}{||\bg^+||}\right)\bg^+\quad \text{otherwise}, 
\end{cases}
\end{equation}
Equivalently, with the convention that $\frac{\bg^+}{||\bg^+||} = 0$ if $\bg^+=0$, we can write,
\begin{equation*}
    \bz = \left(||\bg^+||-\epsilon\sum_{i\in[n]} g_i^-\right)^+\frac{\bg^+}{||\bg^+||}.
\end{equation*}
Notice that $\forall i\in [n] :  g_i^-z_i = 0$ and that $||\bz||\geq ||\bg^+||-\epsilon\sum_i g_i^-$. Thus, for $\epsilon\geq 0$,
\begin{align*}
    \min_{\bz\in K(\epsilon)}||\bg-\bz||^2&\geq \Lcal(\bz,2\bg^-)\\
    &= ||\bg-\bz||^2+\sum_{i\in [n]}2g_i^-(\epsilon||\bz||-z_i)\\
    &= ||\bg-\bz||^2+ \sum_{i\in[n]}2g_i^-\epsilon||\bz||\\
    &\geq ||\bg^-||^2 + 2\epsilon\left(\sum_{i\in [n]} g_i^-\right)||\bg^+||-2\epsilon^2 \left(\sum_{i\in[n]} g_i^-\right)^2.
\end{align*}
Taking expectations, and using $\Ebb(||\bg^+||)\geq \sqrt{\frac{n}{2}}-\frac{3}{\sqrt{n}}$, as well as $\Ebb\left(\left(\sum_{i\in[n]} g_i^-\right)^2\right) = \frac{n}{2}+n(n-1) \frac{1}{2\pi}\leq n^2$, and\footnote{Here we condition on the sign of $g_1$ and then use
Cauchy-Schwartz $\Ebb(||\bg^+||)\geq \frac{1}{\sqrt{n}}\Ebb\left(\onebf^T\bg^+\right) = \sqrt{\frac{n}{2\pi}}$ for the remaining $n-1$ coordinates.} 
$\Ebb(g_1^-||\bg^+||)  = \frac{1}{2}\Ebb_{g_1\leq 0}(g_1^-||\bg^+||)  = \frac{1}{\sqrt{2\pi}}\Ebb\left(\sqrt{\sum_{i=2}^n(g_i^+)^2}\right)\geq \frac{1}{2\pi}\sqrt{n-1}\geq \frac{\sqrt{n}}{16}$ for $n\geq 2$, we get, 
\begin{equation*}
    \Ebb\left(\min_{\bz\in K(\epsilon)}||\bg-\bz||^2\right)\geq \frac{n}{2}+\epsilon\frac{n\sqrt{n}}{8}-2\epsilon^2n^2
\end{equation*}
and thus, 
\begin{align*}
    \rho_1(\epsilon) &\leq \frac{n}{2}-\frac{n\sqrt{n}}{8}+2\epsilon^2 n^2.
\end{align*}

We now show that the minimizer of $\Lcal(\bz,2\bg^-)$ is given by \eqref{eq:minimizer-explicit}. We write the first order condition of the Lagrangian \eqref{eq:lagrangian}, assuming that $\bz\neq 0$ at the optimum and denoting $\bar \lambda = \sum_{i\in [n]}\lambda_i$. We have 
$2(z_i-g_i)+\epsilon \bar\lambda\frac{z_i}{||\bz||} - \lambda_i = 0$,
which implies, $
z_i = g_i +\frac{1}{2}(\lambda_i -\epsilon \bar \lambda \frac{z_i}{||\bz||})$,
and equivalently
\[
z_i\left(2+\frac{\epsilon \bar \lambda}{||\bz||}\right) = 2g_i +\lambda_i.
\]
Recalling the choice of dual variables, $\lambda_i = 2g_i^-\geq 0$ we get,
\begin{equation}\label{eq:expression}
        z_i \left(1+\frac{\epsilon \sum_i g_i^-}{||\bz||}\right)= g_i^+.
\end{equation}
By taking the norm in both sides of \eqref{eq:expression}, we notice that we must have $||\bz||+\epsilon\sum_i g_i^- = ||\bg^+||$, and thus, $||\bz|| = ||\bg^+|| - \epsilon\sum_i g_i^-$. Therefore the first order condition can only be met if $||\bg^+|| - \epsilon\sum_i g_i^- \geq 0$ and otherwise, we must have $\bz = 0$. This recovers Equation \eqref{eq:minimizer-explicit} and finishes the proof.
\end{proof}

\subsection{Proof of Proposition \ref{prop-norm}}
The goal of this section is to show Proposition \ref{prop-norm}, which bounds the norm of the optimal strategy of a random game as follows. 
\propnorm*
\begin{proof}
The uniqueness was shown in Proposition \ref{prop: prelim}.
We denote by $R$ the set of supporting rows, and by $C$ the set of supporting columns of the game, and by $k = |C| = |R|$ their cardinality. All these quantities are well-defined almost surely, as discussed in Proposition \ref{prop: prelim}. In Lemma \ref{lem:supporting-dimension}, which we prove in Section \ref{lem:supporting-dimension}, we show that,
$$\Pbb(k\geq n/20)\geq 1- 2^{-0.3n}.$$ 
We denote by $\Ical = \{(R,C) : |R|=|C| \geq n/20\}$ the set of pairs of row and column indices of submatrices of size greater than $n/20$. Recall from Proposition \ref{prop: prelim}, that if $k\geq n/20$, the optimal strategy for the column player satisfies,
$$||\by|| \in \Ycal = \left\{\frac{1}{\onebf^TQ_{R,C}^{-1}\onebf}||Q_{R,C}^{-1}\onebf|| \quad\text{s.t.} \quad Q_{R,C}^{-1}\onebf\geq 0 \quad \text{and}\quad  (R,C)\in \Ical \right\}.$$
For a fixed pair of indices $(R,C)\in \Ical$ of cardinality $k\leq n$, we show in Lemma \ref{lem:rotation-invariance} that the random vector $Q_{RC}^{-1}\onebf\in \Rbb^k$ is rotationally invariant in $\Rbb^k$ (this follows from the definition of the Haar measure). Therefore, 
\begin{equation*}
    \Pbb\left(\frac{1}{\onebf^T Q_{RC}^{-1}\onebf}||Q_{RC}^{-1}\onebf||\geq \frac{c_0}{ \sqrt{k}}\quad \&\quad Q_{RC}^{-1}\onebf\geq 0\right) = \Pbb\left(\onebf^T\bz \leq \frac{1}{c_0}\sqrt{k}\quad \&\quad \bz\geq 0\right) \leq 2^{-3n},
\end{equation*}
where $\bz = \frac{1}{||\bg||}\bg$, for $\bg\sim \Ncal(0,I_k)$, and where $c_0>0$ is the appropriate constant in Lemma \ref{lem:laurent} below, to get the rate of $2^{-3n}$. 
\begin{restatable}{lemma}{lemlaurent}\label{lem:laurent}
For $\bg\sim \cN(0,I_k)$, and $\bz = \frac{1}{||\bg||}\bg$, one has:
$$
\dP(\bz\ge 0~\&~\onebf^\top \bz\le \theta \sqrt{k})\le e^{-kf(\theta)}
$$
where $\lim_{\theta\to 0}f(\theta)=+\infty$.
\end{restatable}
Since the cardinality of 
$\Ycal$ is less than $2^{2n}$, we can conclude by union bound that 
\begin{equation*}
    \Pbb\left(||\by||\geq \frac{\sqrt{20}~c_0}{\sqrt{n}}\right)\leq 2^{-n}+2^{-0.3 n},
\end{equation*}
which finishes the proof. The constants $c_1,c_2>0$ are chosen appropriately.
\end{proof}

\subsection{Proof of Lemma \ref{lem:supporting-dimension}}
In this section, we prove a series of lemmas, culminating with Lemma \ref{lem:supporting-dimension} on the norm of the optimal strategy. 
\begin{lemma}[Orthogonal invariance]\label{lem:orthogonal-inv} We consider the block decomposition of a matrix $Q\in \Rbb^{n\times n}$
$$Q = \begin{pmatrix}
Q_1 & Q_2 \\
Q_3 & Q_4
\end{pmatrix}$$
with $Q_1$ of shape $k\times k$ and $Q_4$ of shape $(n-k)\times (n-k)$. For any orthogonal matrices $U_1\in O_k$, $U_2\in O_{n-k}$ and $U_3 \in O_{n-k}$, if $Q\in \Rbb^{n\times n}$ is distributed from the Haar measure on orthogonal matrices, i.e., $Q\sim \Ucal(O_n)$, we have the following equalities in distribution,
\begin{align*}
    Q_1=^dU_1Q_1,\quad &\text{and}\quad Q_1 U_1 =^d Q_1,\\
    (Q_1,Q_2U_2) &=^d (Q_1, Q_2),\\
    (Q_1,U_3 Q_3)&=^d(Q_1,Q_3).
\end{align*}
\end{lemma}
\begin{proof}
We shall prove the third statement; the other two are very similar. Consider $U_3\in O_{n-k}$ and $U\in O_n$ defined by $$U = \begin{pmatrix}
I_k & 0 \\
0 & U_3
\end{pmatrix}.$$ 
From the definition of the Haar measure, one has, $UQ=^dQ$. The block product writes as
    \[
UO = \begin{pmatrix}
I_k & 0 \\
0 & U_3
\end{pmatrix}
\begin{pmatrix}
Q_1 & Q_2 \\
Q_3 & Q_4
\end{pmatrix}
= \begin{pmatrix}
Q_1 & Q_2 \\
U_3 Q_3 & U_3 Q_4
\end{pmatrix}.
\]
and thus proves that $(Q_1, U_3 Q_3)=^d(Q_1,Q_3)$.
\end{proof}

\begin{lemma}[Rotational invariance]\label{lem:rotation-invariance}
    Let $R,C\subset [n]$ be fixed subsets of same cardinality $k\geq 1$ and $Q\sim \Ucal(O_n)$, then the vector $\bz = (Q_{RC})^{-1}\onebf$ is rotationally invariant in $\Rbb^k$. In other terms, the direction of $\bz$ is uniform and independent from its norm.
\end{lemma}
\begin{proof}
By Lemma \ref{lem:orthogonal-inv}, for any fixed $U\in O_k$ we have $Q_{RC}U^T =^d Q_{RC}$. Taking the inverse in this equality, we have $U(Q_{RC})^{-1} =^d (Q_{RC})^{-1}$, i.e., that $(Q_{RC})^{-1}$ is orthogonally invariant. Thus, for any $U\in O_k$, we have $U(Q_{RC})^{-1}\onebf =^d (Q_{RC})^{-1}\onebf$. This equality in distribution, which hold for all $U\in O_k$, remains true when $U$ is sampled from the Haar measure, $U\sim \Ucal(O_k)$, thus the lemma. 
\end{proof}

\begin{lemma}\label{lem:top-left}
Let $R,C\subset [n]$ be fixed subsets of the same cardinality $k\geq 1$ and $Q$ be a random matrix from the Haar measure, the probability that the game on $Q$ is supported by rows $R$ and columns $C$ is less than $2^{k-n}$. 
\end{lemma}
\begin{proof}
The following argument borrows from the work of Jonasson \cite{jonasson2004optimal}. Without loss of generality, we study the probability that the game is supported by the top-left submatrix $Q_1\in \Rbb^{k\times k}$,  where $$Q = \begin{pmatrix}
Q_1 & Q_2 \\
Q_3 & Q_4
\end{pmatrix}.$$
By Proposition \ref{prop: prelim}, there are two necessary conditions for the game on $Q$ to be supported by $Q_1$:
\begin{enumerate}[label=(\roman*)]
    \item The vectors $\bx = \frac{1}{\onebf^T \left(Q_1^{T}\right)^{-1}\onebf}\left(Q_1^{T}\right)^{-1} \onebf$ and $\by = \frac{1}{\onebf^T Q_1^{-1}\onebf}Q_1^{-1}\onebf$, satisfy $\bx\geq 0$ and $\by\geq 0$. \label{cond:equa}
    \item These vectors (also called equalizing strategies) satisfy the saddle point condition $\bx^TQ_2 \leq v\onebf\leq Q_3\by$ where $v = \bx^T Q_1 \by$.\label{cond:stab}
\end{enumerate}
Condition \ref{cond:stab} reflects that the column player does not benefit from playing on $Q_2$ when the row player uses its equalizing strategy, and that the row player does not benefit from playing on $Q_3$ when the column player uses its equalizing strategy. 

The probability that the game is supported by its $k\times k$ top-left submatrix is thus less than
$$\Pbb(\ref{cond:stab} \text{ is satisfied})  = \Pbb(\bx^T Q_2 \leq v\onebf \leq Q_3 \by).$$
We now condition on the value of $Q_1$ and we shall bound $\Pbb(\ref{cond:stab} \text{ is satisfied} |Q_1)$. Note that the values of $v,\bx,\by$ are fixed under this conditioning. We assume without loss of generality that $v\geq 0$ and we write
$$\Pbb(\ref{cond:stab} \text{ is satisfied} |Q_1) \leq \Pbb(v\onebf\leq Q_3\by |Q_1).$$
Since we have by Lemma \ref{lem:orthogonal-inv} that $(Q_1,Q_3) =^d(Q_1,U_3Q_3)$ for any fixed orthogonal matrix $U_3\in\Rbb^{(n-k)\times (n-k)}$, we can take $U_3$ to be a random diagonal matrix with diagonal values sampled independently and uniformly from $\{+1,-1\}$. Note that all entries of $U_3 Q_3\by$ now have random signs. Because $v$ is positive, we have $v\leq U_3Q_3\by$ only if $U_3Q_3\by \in \Rbb^{n-k}$ is a vector with all signs positive, which happens with probability $2^{k-n}$. We treat the case $v\geq 0$ in the same way, using that $(O_1,O_2)=^d (O_1,O_2 U_2)$ for $U_2\in O_2$. We have shown that, irrespective of the value taken by $Q_1$, we have $\Pbb(\ref{cond:stab} \text{ is satisfied} |Q_1) \leq 2^{k-n}$. Thus, we can conclude,
$$\Pbb(\ref{cond:stab} \text{ is satisfied}) \leq 2^{k-n},$$
which finishes the proof.
\end{proof}

\begin{lemma}\label{lem:supporting-dimension}
The cardinality $k$ of the support of the game on $Q\sim \Ucal(O_n)$, satisfies 
$$\Pbb(k\leq n/20)\leq 2^{-0.3 n}.$$  
\end{lemma}
\begin{proof}
By Lemma \ref{lem:top-left}, we observe that the probability that a given submatrix of size $k\times k$ is the support of the game is bounded by $2^{k-n}$. Therefore, by union bound, the probability that any submatrix of size smaller than $n/20$ is the support of the game is less than,
\begin{align*}
    2^{-19/20n}\sum_{k\leq n/20}\binom{n}{k}^2&\leq 2^{-19/20n}\sum_{k\leq n/20}2^{2nH(k/n)}\\
    &\leq \frac{n}{20}2^{(2H(1/20)-19/20)n}\\
    &\leq 2^{-0.3 n},
\end{align*}
where we used that $\binom{n}{k}\leq 2^{nH(k/n)}$ and where $H(p)=-p\log_2(p)-(1-p)\log_2(1-p)$ is the binary entropy function (which is positive and satisfies $H(1/20) \leq 0.3$). 
\end{proof}

\subsection{Proof of Lemma \ref{lem:laurent}}
In this section, we prove the following concentration lemma. 
\lemlaurent*

\begin{proof}

Write for $\gamma <1/2$:
$$
\dE (e^{\gamma g_1^2})=\int_{-\infty}^\infty \frac{1}{\sqrt{2\pi}}e^{-(1-2\gamma)x^2/2}dx=(1-2\gamma)^{-1/2}.
$$
Thus
$$
\dP(\|\bg\|^2\ge k u)\le \exp(-k h(u)),
$$
where for $u>1$, one has
$$
h(u)=\sup_{\gamma\in \Rbb} \gamma u +\frac{1}{2}\ln(1-2\gamma)=\frac{1}{2}[u-1+\ln(1/u)].
$$
Write then, for $r>0$,
\begin{align*}
    \dP(\bz\ge 0~\&~\onebf^T \bz \le v \sqrt{k})&\leq \Pbb(\onebf^T|\bz|\leq v\sqrt{k})\\
    &=\dP(\|\bg\|^{-1}\onebf^T |\bg|\le v\sqrt{k})\\
&\le \dP(\|\bg\|\ge r\sqrt{k})+\dP(\onebf^T |\bg|\le r v k).
\end{align*}
Write now for $\gamma>0$:
$$
\dE(e^{-\gamma |g_1|})=\sqrt{2/\pi}\int_0^\infty e^{-\gamma x-x^2/2}dx.
$$
We upper bound twice the Gaussian density on $[0,2]$ by some constant $\alpha>0$, and upper bound it by $e^{-x}$ on $[2,\infty)$, thus obtaining
$$
\dE(e^{-\gamma |g_1|})\le \alpha\frac{1-e^{-2\gamma}}{\gamma} +\int_2^\infty e^{-(\gamma+1)x}dx
$$
hence
$$
\dE(e^{-\gamma |g_1|})\le \alpha\frac{1-e^{-2\gamma}}{\gamma}+\frac{e^{-2(\gamma+1)}}{\gamma+1}.
$$
Finally we obtain 
$$
\dE(e^{-\gamma |g_1|})\le \frac{c_0}{\gamma}
$$
for some constant $c_0>0$.
Thus:
$$
\dP(\onebf^T |\bg|\le \beta k)\le \exp(-kf(\beta)),
$$
where
$$
f(\beta)=\sup_{\gamma>0} (-\gamma \beta -\ln \dE (e^{-\gamma |g_1|})).
$$
Thus:
$$
f(\beta)\ge \sup_{\gamma>0}- \gamma \beta -\ln(c_0/\gamma)=-1-\ln(c_0)+\ln(1/\beta).
$$
This yields
$$
\dP(\bz\ge 0~\&~\onebf^T \bz \le v \sqrt{k})\le e^{-k h(r^2)}+e^{-k f(rv)}.
$$
Now, since 
$$
\lim_{r\to +\infty}h(r)=\lim_{\epsilon\to 0}f(\epsilon)=+\infty,
$$
we thus have that
$$
\dP(\bz\ge 0~\&~\onebf^T \bz \le v \sqrt{k})\le 2 e^{-k \psi(v)},
$$
where $\lim_{v\to 0}\psi(v)=+\infty$.

Indeed, this follows from the previous bound by setting $r=v^{-1/2}$, and 
$$
\psi(v)=\min(h(v^{-1}),f(v^{1/2})).
$$
\end{proof}

\section{Open directions and conclusion}\label{sec:open-directions}
\subsection{Open directions}
\paragraph{Other distributions and related problems.} The scaling described in Theorem \ref{th:main} appears to extend beyond continuous distributions, such as, for example, to matrices where all entries are independent Bernoulli random variables, see Figure \ref{fig:scaling}. 
This generalization is particularly significant in computer science, where the costs of underlying zero-sum games are usually discrete (e.g., representing a number of operations). 
Note that settings where the matrix $M$ takes $\{0,1\}$ values have alternative interpretations as problems on random bipartite graphs $G=(V,E)$, where the nodes are partitioned into two subsets $V = [n]\cup [m]$ and where each potential edge $(i,j)\in [n]\times [m]$ is in $E$ if and only if $M_{ij}=1$. Specifically, the following problems are equivalent (up to a reduction): 
\begin{enumerate}
    \item \textit{A hide-and-seek game in a random bipartite graph:} consider the game where player 1 picks a hiding strategy (i.e., distribution) over the nodes in $[n]$ and player 2 picks a node in $[m]$ to maximize the chance of being connected to player 1 by an edge. If both players use optimal strategies, the probability that the seeker wins the game is equal to $v(M)$;
    \item \textit{Fractional vertex covering in a random bipartite graph:} consider the task of defining a fractional configuration on the nodes of 
    $[m]$ denoted by $\by\in \Rbb_+^m$ such that it covers the left side nodes, in the sense that $\forall i\in [n]: \sum_{(i,j)\in E}y_j\geq 1$ while minimizing the cost $\sum_{j\in [m]} y_j$. The weight of the optimal fractional covering is equal to $1/v(M)$. 
\end{enumerate}
This setting with Bernoulli entries is also related to other problems in computer science, such as discrepancy minimization, set cover, and assignment problems \cite{pesenti2023discrepancy,vazirani2001approximation,aldous2001zeta}. 

Despite the empirical evidence provided in Figure \ref{fig:scaling}, our techniques do not naturally extend to discrete settings --- and Gaussian comparison inequalities are well known to be difficult to generalize. However, we believe that other methods could be used in the sparse regime, where $M_{ij}\sim \Bcal(d/n)$ for a fixed average degree $d = \Ocal(1)$, because the corresponding bipartite graph is locally tree-like (see, e.g., \cite{aldous2004objective}). 

\paragraph{Computationally efficient algorithms for finding near-optimal strategies.} A surprising aspect of our proofs is that they are non-constructive, in the sense that we do not exhibit a strategy $\bx\in \Delta_n$ that is within the standard deviation of the optimal strategy. One strategy that is explicit and computationally inexpensive is the uniform strategy, where $\bx = \frac{1}{n}\onebf$, but it is now clear that this strategy is suboptimal. A natural question is whether there exists a \textit{simple} strategy that achieves a value that is close from the optimal value with high probability---i.e., that benefits from \textit{the advantage of rationality} \cite{faris1987value}. By \textit{simple strategy}, we mean a strategy that has a closed form, or alternatively, that is given by an iterative procedure requiring less computational effort than obtaining a $\Ocal\left(\frac{1}{n}\right)$-approximation of the optimal strategy with off-the-shelf linear optimization algorithms.

\paragraph{Structure of optimal strategies.} One subject of attention in some previous works on random games \cite{jonasson2004optimal,brandl2017distribution} was the structure of the optimal distribution $\bx\in \Delta_n$, with a special focus on the cardinality of its support. In the Gaussian setting that we describe, it has been empirically observed that the support of $\bx$ is very close to being uniform over the subsets of $[n]$ \cite{faris1987value}, although that is not exactly the case (e.g., note that the support of $\bx$ cannot be empty).
Jonasson \cite{jonasson2004optimal} observed that results on the question would be tightened once the value of the game would be shown to be in $\Theta(1/n)$ (see his Conjecture 2.9 and its implications). In an attempt at going further, reading the proof of Theorem \ref{th:main}, it is tempting to believe that the pair of optimal strategies $(\bx,\by)$ for a random game has a structure that is an analog to the structure of the pair $\left(\frac{\bg^-}{\onebf^T\bg^-},\frac{\bh^+}{\onebf^T\bh^+}\right)$, where $\bg\sim \Ncal(0,I_n)$ and $\bh\sim \Ncal(0,I_m)$. However, this intuition leaves some glaring gaps: in particular, the cardinalities of the supports of $\bx$ and $\by$ should be equal (Proposition \ref{prop: prelim}), whereas $\bg$ and $\bh$ are independent vectors. Thus, we believe that fundamentally new arguments are needed to precisely capture the structure of optimal strategies. 

\paragraph{Interpretations in algorithmic complexity.} In our opinion, one of the most exciting possible developments would be to revisit classical questions in algorithmic complexity under the lens of random games. In light of the statistical physics interpretation presented in the introduction, we are tempted to argue that the framework could provide heheuristic motivation for the universality of certain phenomena in algorithmic complexity (e.g., NP completeness in computational complexity, or the exponential advantage conferred by randomization in online algorithms). We therefore go back to the framework of Yao, where we consider the matrix associated with an algorithmic problem, in which rows are indexed by algorithms and columns are indexed by inputs, and where the values of the matrix are given by the cost of an algorithm on an input. Of course, there is often no natural enumeration of algorithms and instances for a given algorithmic problem, and part of the challenge would be to find the right problems for which the corresponding matrix is accurately modeled by a random matrix. In the following paragraph, we provide a high-level and informal discussion on a classical computational complexity question under the lens of random games. 

Algorithmic problems that are solvable in polynomial time with randomness ($RP$) are widely conjectured to be solvable in polynomial time by a deterministic algorithm ($P$). This conjecture, denoted $P=RP$, should come as a surprise after reading this paper because in typical zero-sum games we saw that mixed strategies have a significant advantage over pure strategies (cf. discussion in the related work on matrices with saddle-points).  Therefore, this suggests that matrices that describe the complexity of algorithmic problems in $P$ are all very structured (far from square matrices with i.i.d. entries). Specifically, either they are very skewed $(n>>m)$ with many more rows (possible algorithms) than there are columns (possible instances), or, alternatively, they satisfy structural properties that strongly violate the independence assumption, such as that given two rows there exists a third row that interpolates between these two (effectively allowing to mimic a mixed strategy with one pure strategy). Thus, we believe that understanding the class of random matrices for which we have $\min_{\bx\in \Delta_n}\max_{j\in [m]}(\bx^T M)_j \approx \min_{i\in[n]} \max_{j\in [m]} M_{ij}$ could lead to insights on the structure of the class of algorithmic problems that can be solved in polynomial time. More generally, we believe that random games provide a natural model for exploring the magnitude of the advantage that randomization can confer to algorithms and decision rules.

\subsection{Conclusion}
In this paper, we give new tail bounds on the value of a random game. Our results rely on Gaussian comparison principles and on convex geometry. We believe that there are promising avenues for research in random games at the intersection of mathematics, statistical physics, and computer science, and we have highlighted some open directions above. 

\paragraph{Acknowledgments.} RC thanks Moïse Blanchard and Vianney Perchet for helpful discussions. 

\newpage
\appendix

\section{Experiments and simple concentration results}
\subsection{Empirical evidence for the scaling of the game value.}\label{ap:experiment}
Below we provide some empirical evidence of the scaling of the value of $v(M)$ for square matrices. Although our theoretical results only hold for Gaussian matrices, Figure \ref{fig:scaling} also provides empirical evidence that the scaling of the game value generalizes to matrices with independent Rademacher entries (i.e., taking value $\pm 1$ with probability $1/2$). Note that the original experiments leading to the conjecture that the standard deviation of $v(M)$ is $\Theta(1/n)$ is due to \cite{faris1987value}. We refer to their work for more thorough discussion on experimental settings. 
\begin{figure}[h]
    \centering
\includegraphics[width=0.6\linewidth]{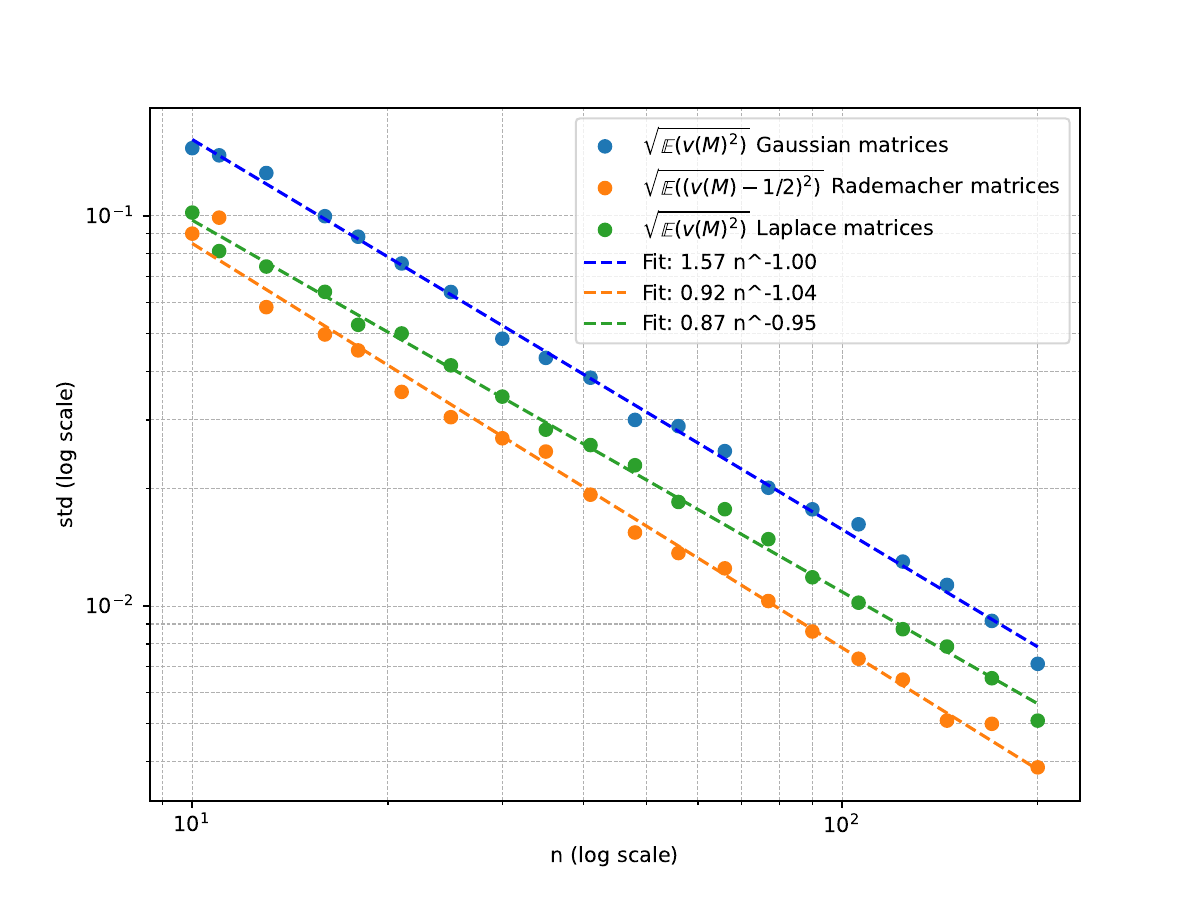}
    \caption{Empirical scaling of random game value $v(M)$ for square matrices with size $n$ going from $5$ to $200$. Each data point corresponds to an experiment with batch size $200$.
    }
    \label{fig:scaling}
\end{figure}
\subsection{Lower bound on $\Ebb(||\bg^+||)$}\label{ap:technical-binom}
We note that by conditioning on the number of positive components of $\bg$, we have $\Ebb(||\bg^+||) = \Ebb(a_X)$, where $X\sim\Bcal(1/2,n)$, and where $a_n$ is the expectation of the chi distribution with $n$ degrees of freedom. Then, using \eqref{eq:bound-a}, we obtain, 
\begin{equation}\label{eq:lb-binom}
    \Ebb(||\bg^+||) \geq \Ebb\left(\sqrt{X}-\frac{\one_{X>0}}{2\sqrt{X}}\right),
\end{equation}
and we now analyze the right-hand side of Equation \eqref{eq:lb-binom}. 

We first derive a lower bound on $\Ebb(\sqrt{X})$. Noticing that $\forall x\in \Rbb_+ : \sqrt{x}\geq 1+\frac{x-1}{2}-\frac{(x-1)^2}{2}$ and applying this inequality to $X/\mu$, for $\mu = n/2$, we get, using that $\Ebb(X/\mu) = 1$ and that $\Var(X/\mu) = \frac{\Var(X)}{\mu^2} = \frac{1}{2\mu}$, we get, $\Ebb\left(\sqrt{\frac{X}{\mu}}\right) \geq 1- \frac{1}{4\mu}$, thus, 
\begin{align}\label{eq:lb1}
\Ebb(\sqrt{X}) &\geq \sqrt{\mu}-\frac{1}{4\sqrt{\mu}} = \sqrt{\frac{n}{2}}-\frac{1}{2\sqrt{2n}}.
\end{align}

Then, we provide an upper bound on $\Ebb\left(\frac{1}{\sqrt{X}}\oneb_{X>0}\right)$. By Chernoff bounds we have that $\Pbb(X\leq n/4)\leq \exp\left(-\frac{1}{8}n\right)$ and thus, 
\begin{equation}\label{eq:lb2}
    \Ebb\left(\frac{1}{\sqrt X} \oneb_{X>0}\right) \leq \exp\left(-\frac{1}{8}n\right) + \frac{2}{\sqrt{n}}\leq \frac{4}{\sqrt{n}}.
\end{equation}  
Combining Equations \eqref{eq:lb-binom}, \eqref{eq:lb1}, \eqref{eq:lb2} entails $\Ebb(||\bg^+||)\geq \sqrt{\frac{n}{2}}-\frac{3}{\sqrt{n}}$.

\subsection{Consequence on expected value of rectangular matrices}\label{ap:expectation-rectangular}
We explain here how our Theorem \ref{th:main} implies that, for $n<m\leq 2n$, we have 
    \begin{equation*}
        \Omega\left(\frac{\sqrt{m}-\sqrt{n}}{n}\right)-\Ocal\left(\frac{1}{n}\right) \leq \Ebb(v(M))\leq \Ocal\left(\frac{\sqrt{m}-\sqrt{n}}{n}\right)+\Ocal\left(\frac{1}{n}\right).
    \end{equation*}
For the upper bound, we consider the following decomposition, where we denote for shorthand $A = 40\frac{\sqrt{m}-\sqrt{n}}{n}$, and apply the tail bounds of Theorem \ref{th:main} and Proposition \ref{prop:naive-uniform}, 
\begin{align*}
      \Ebb(v(M)) &\leq  \int_{0}^\infty \Pbb(v(M)\geq s)ds\\
      &\leq \int_0^A \Pbb(v(M)\geq s)ds +\int_{0}^{1-A} \Pbb(v(M)\geq A+s)ds+ \int_{1}^\infty \Pbb(v(M)\geq s)ds\\
      &\leq A + \int_0^{1-A}c_1\exp(-c_2n^2s^2)+c_3\exp(-c_4 n)ds+\int_{1}^\infty m\exp(-ns^2)ds\\
      &\leq A+\Ocal\left(\frac{1}{n}\right),
    \end{align*}
    where we used by change of variable $\int_{0}^\infty\exp(-c_2n^2 s^2)ds = \frac{1}{n\sqrt{c_2}}\int_{0}^\infty\exp(-u^2)du = \Ocal(\frac{1}{n})$ and the direct computation $\int_{1}^\infty\exp(-ns^2)ds \leq \int_1^\infty \exp(-ns)ds = \exp(-n)/n$.

    For the lower bound, we use the shorthand $B = \frac{1}{4}\frac{\sqrt{m}-\sqrt{n}}{n}-\frac{\sqrt{\ln(c_1)}}{n\sqrt{c_2}}$ to get,
    \begin{align*}
        \Ebb(v(M))&\geq \int_0^{B}\Pbb(v(M)\geq s)ds - \int_{0}^\infty\Pbb(v(M)\leq -t)dt\\
        &\geq B\left(1 - e^{-1} - c_3\exp(-c_4n))\right)-\Ocal\left(\frac{1}{n}\right)\\
        &\geq \Omega\left( \frac{\sqrt{m}-\sqrt{n}}{n}\right)-\Ocal\left(\frac{1}{n}\right).
    \end{align*}
We note that the slightly unconventional $\Omega(\cdot)-\Ocal(\cdot)$ notation here reflects the fact that for $m = n+ cst$, e.g., for $m=n+1$, we do not argue about that there is a lower bound of $\Ebb(v(M))$ in $\Omega(1/n)$.

\subsection{Exploiting the symmetry in optimization problem \eqref{eq:lp-init}}\label{ap:symmetry}
Using that $v(M) = -v(-M^T) =^{dist} -v(M^T)$, we get that, 
$$\Pbb(v(M)\geq t) = \Pbb(v(M^T)\leq -t).$$
Then, we recall the definition $\Phi_2(\bg,\bh) = \min_{\bu\in \Delta_n}\max_{\bv\in \Delta_m}  ||\bv||\bg^T\bu+||\bu||\bh^T\bv$ and also define $\Psi_2(\bg,\bh) = \min_{\bv\in \Delta_m}\max_{\bu\in \Delta_n}||\bv||\bg^T\bu+||\bu||\bh^T\bv$. We can use the basic minimax inequality to write
\begin{equation*}
    -\Psi_2(-\bg,-\bh) = \max_{\bv\in \Delta_m}\min_{\bu\in \Delta_n} ||\bv||\bg^T\bu+||\bu||\bh^T\bv \leq \Phi_2(\bg,\bh).
\end{equation*}
Thus, also using $(-\bg,-\bh)=^{dist}(\bg,\bh)$ we have that
\begin{equation*}
    \Pbb(\Phi_2(\bg,\bh)\geq t)\geq \Pbb(-\Psi_2(\bg,\bh)\geq t).
\end{equation*}
We can now conclude by applying Equation \eqref{eq:gordon-app1} to the transposed matrix $M^T\in \Rbb^{m\times n}$, obtaining an upper tail bound for $M$ in the form of 
\begin{equation*}
    \Pbb(v(M)\geq t) = \Pbb(v(M^T)\leq -t) \leq 2 \Pbb(\Psi_2(\bg,\bh)\leq -t)\leq 2\Pbb(\Phi_2(\bg,\bh)\geq t).
\end{equation*}
Note that some statements of Gordon's inequality explicitly mention the minimax properties of $\Phi_2(\bg,\bh)$ in the theorem statement, e.g., \cite{miolane2021distribution}. 

\subsection{Concentration results used in the proof of Theorem \ref{th:main}}\label{sec:concentration}
We conclude the proof of Theorem \ref{th:main} by providing some details on the concentration inequalities that we applied along the way. In the following, when we write, \textit{with high probability} (abbreviated w.h.p.), we mean, with probability at least $1-\Ocal(\exp(-c_4 n))$ for some positive constant $c_4$. 

\paragraph{Bounds on $||\bh^+||, \onebf^T\bh^+$ and on $||\bg^-||,\onebf^T\bg^-$.} All the inequalities on those terms are based on the Gaussian concentration inequality introduced in the preliminaries. Note that we use that $\Ebb(||\bh^+||)\in [\sqrt{m/2}-3/\sqrt{m},\sqrt{m/2}]$ and $\Ebb(\onebf^T\bh^+) = m/\sqrt{2\pi}$. 

\paragraph{Bound on $C'(\mu')$.} We have that $C'(\mu') = \onebf^T(-\bg-\mu'\onebf)^+$, thus by Cauchy-Schwarz, $C'(\mu')\leq ||\onebf|||(-\bg-\mu'\onebf)^+||= \sqrt{n}||\bh^+||$. Thus, with high probability, $C'(\mu')\leq \sqrt{nm}\leq \sqrt{2}n$. 

\paragraph{Bounds on $C(\mu)$.} We first show that $\mu\leq 1$ with high probability. For this, observe that $\Ebb(||(\bh-\onebf)^+||)\leq \sqrt{\Ebb(||(\bh-\onebf)^+||^2)}\leq \sqrt{\frac{m}{5}}\leq \sqrt{\frac{2n}{5}}$, where we use that $\Pbb(\Ncal(0,1)\geq 1)\leq \frac{1}{5}$ and $m\leq 2n$.
Then, we use that $\sqrt{\frac{n}{2}}-\frac{3}{\sqrt n}\leq \Ebb(||\bg^-||)$ and thus have that $||(\bh-\onebf)^+||\leq ||\bg^-||$ w.h.p, and thus, $\mu \leq 1$ w.h.p. Thus, $C(1)\leq C(\mu)$ w.h.p., where $\Ebb(C(1)) = m\Ebb((h_1-1)^+)=\varphi(1)-(1-\Phi(1))>0.083m$. By Gaussian concentration inequalities, w.h.p., $C(1)\geq m/20$ and therefore, w.h.p. $C(\mu)\geq m/20$.

\section{Gordon's inequality and Gaussian comparison principles}\label{sec:gaussian-comparison}
In this section, we provide some background on the Gaussian comparison principles, which we used in Section \ref{sec:main}. We start by stating the version of Gordon's inequality exactly as we use it in this paper, then, for reader's convenience, we provide a proof sketch of the result. The following version of Gordon's inequality is an analog of Corollary G.1 in \cite{miolane2021distribution}.

\thgordon*
\begin{proof}[Proof sketch] We first focus on proving Equation \eqref{eq:base-eq} in the case where $D_x$ and $D_y$ are finite. We define two Gaussian vectors indexed by $D_x\times D_y$ by 
\begin{align*}
    X(\bx,\by) &= \bx^T G\by+z||\bx||||\by||,\\
    Y(\bx,\by) &= ||\by||\bg^T\bx + ||\bx||\bh^T\by,
\end{align*}
where $z\sim\Ncal(0,1)$ is an independent standard Gaussian scalar. Direct computation for any $\bx,\bx',\by,\by'\in D_x^2\times D_y^2$ yields,
\begin{align*}
    \Ebb(X(\bx,\by)X(\bx',\by')) &= (\bx^T \bx')(\by^T \by') + ||\bx||||\by||||\bx'||||\by'||,\\
    \Ebb(Y(\bx,\by)Y(\bx',\by')) &= ||\by||||\by'||\bx^T\bx' + ||\bx||||\bx'||\by^T\by'.
\end{align*}
Thus, using $(\bx^T\bx'-||\bx||||\bx'||)(\by^T\by'-||\by||||\by'||)\geq 0$, we get,
\begin{align*}
    \Ebb(X(\bx,\by)X(\bx',\by')) \geq  \Ebb(Y(\bx,\by)Y(\bx',\by')),
\end{align*}    
with equality when $\bx = \bx'$. Then applying Gordon's inequality (as in Theorem \ref{th:gordon-orignal} below) provides for all $t\in \Rbb$,
\begin{equation*}
    \Pbb\left(\min_{\bx\in D_x}\max_{\by\in D_y}X(\bx,\by)\leq t\right)\leq \Pbb\left(\min_{\bx\in D_x}\max_{\by\in D_y}Y(\bx,\by)\leq t\right).
\end{equation*}
Conditioning on the event $z\leq 0$, of probability $1/2$, we have that  $\Pbb\left(\min_{\bx\in D_x}\max_{\by\in D_y}X(\bx,\by)\leq t\right)\geq \frac{1}{2}\Pbb_{z\leq 0}\left(\min_{\bx\in D_x}\max_{\by\in D_y}X(\bx,\by)\leq t\right)\geq \frac{1}{2}\Pbb(\min_{\bx\in D_x}\max_{\by\in D_y}\bx^TG\by\leq t)$, thus,
\begin{equation}\label{eq:final}
    \Pbb\left(\min_{\bx\in D_x}\max_{\by\in D_y}\bx^T G\by \leq t\right)\leq 2\Pbb\left(\min_{\bx\in D_x}\max_{\by\in D_y}Y(\bx,\by)\leq t\right).
\end{equation}
Notice that Equation \eqref{eq:final} imposes no constraint on the cardinality of $D_x$ and $D_y$. The initial statement \eqref{eq:base-eq}, which allows for infinite compact sets $D_x,D_y$ is then naturally derived by taking finite coverings of increasing granularity of $(D_x\times D_y)$ (see, e.g., \cite{thrampoulidis2015regularized}). 
\end{proof}
\begin{remark}[Historical perspective]
The original exposition of Gordon's inequality (Theorem~1.1 in \cite{gordon1985some}, 1985) concerns a finite family of Gaussian variables (Theorem \ref{th:gordon-orignal} below). This original work already contains some other formulations (e.g., the case where $D_x$ is the unit sphere as in Corollary 1.2 in \cite{gordon1985some}). The generalization to arbitrary compact sets is, to the best of our knowledge, due to \cite{thrampoulidis2015regularized}.
\end{remark}

\begin{theorem}[Gordon's inequality (Theorem 1.1 in \cite{gordon1985some})]\label{th:gordon-orignal}
Let $(X_{i,j})$ and $(Y_{i,j})$, with $i\in[n]$, $j\in [m]$, be two centered Gaussian random vectors such that
\[
\begin{cases}
\mathbb{E}[X_{i,j}^2] = \mathbb{E}[Y_{i,j}^2], & \text{for all } i, j, \\[4pt]
\mathbb{E}[X_{i,j} X_{i,k}] \le \mathbb{E}[Y_{i,j} Y_{i,k}], & \text{for all } i, j, k, \\[4pt]
\mathbb{E}[X_{i,j} X_{l,k}] \ge \mathbb{E}[Y_{i,j} Y_{l,k}], & \text{for all } i \ne l \text{ and all } j, k.
\end{cases}
\]
Then, for all real numbers $(\lambda_{i,j})_{i,j\in[n]\times [m]}$,
\[
\mathbb{P}\!\left( \bigcap_{i=1}^n \bigcup_{j=1}^m \{ X_{i,j} \ge \lambda_{i,j} \} \right)
\ \ge\
\mathbb{P}\!\left( \bigcap_{i=1}^n \bigcup_{j=1}^m \{ Y_{i,j} \ge \lambda_{i,j} \} \right).
\]
In particular, for all real $t$, 
$$\Pbb\left(\min_i \max_j X_{i,j}\leq t\right)\leq \Pbb\left(\min_i \max_j Y_{i,j}\leq t\right).$$
\end{theorem}
\begin{proof}[Proof sketch.] Although the original proof in \cite{gordon1985some} is involved, this theorem can be obtained as a direct application of Kahane's inequality \cite{kahane1986inegalite} (Theorem \ref{th:kahane} below) to $N  = n\times m$, where elements of $[N]$ are indexed by pairs $(i,j)\in[n]\times[m]$ using the Euclidean division by $m$, and to the function $f(\cdot)$ defined by $\forall x\in \Rbb^N : f(x) = \one\left( \bigcap_{i=1}^n \bigcup_{j=1}^m \{ x_{(i,j)} \ge \lambda_{i,j} \} \right)$. Note that the assumptions of Kahane's inequality are met because the cross derivatives (in the sense of distributions) of $f(\cdot)$ satisfy $\partial_{(i,j)}\partial_{(l,k)} f(x)\geq 0$ for $i\neq l$,  and $\partial_{(i,j)}\partial_{(i,k)} f(x)\leq 0$.   
\end{proof}

\begin{theorem}[Kahane's inequality \cite{kahane1986inegalite}]\label{th:kahane}
Let $X = (X_1, \dots, X_N)$ and $Y = (Y_1, \dots, Y_N)$ be Gaussian random variables in $\mathbb{R}^N$.  
Assume that $A$ and $B$ form a partition of $\{1, \dots, N\} \times \{1, \dots, N\}$ such that 
\begin{align*}
    \mathbb{E}[X_i X_j] \le \mathbb{E}[Y_i Y_j] \quad \text{if } (i,j) \in A,\quad \text{and},\quad 
    \mathbb{E}[X_i X_j] \ge \mathbb{E}[Y_i Y_j] \quad \text{if } (i,j) \in B.
\end{align*}

Let $f$ be real a function on $\mathbb{R}^N$ whose second derivatives (in the sense of distributions) satisfy
\[
\partial_{ij} f \ge 0 \quad \text{if } (i,j) \in A, \quad \text{and},\quad  \partial_{ij} f \le 0 \quad \text{if } (i,j) \in B.
\]
Then
\[
\mathbb{E}[f(X)] \le \mathbb{E}[f(Y)].
\]
\end{theorem}
\begin{proof}[Proof sketch.] Consider two independent copies of vector $X$ and $Y$, as well as the interpolating random variable $Z(t) = \sqrt{1-t}X+\sqrt{t}Y$. The goal is to show that $\gamma(0)\leq \gamma(1)$ where 
$$\gamma(t) = \Ebb(f(Z(t))).$$
It suffices to show that all the terms in the sum $\gamma'(t) = \sum_{i\in [n]}\Ebb(\partial_i f(Z(t))\partial_t Z_i(t))$ are positive. For fixed $t\in[0,1],i\in[N]$, $(Z(t),\partial_t Z_i(t))$ is a Gaussian vector such that the sign of $\Ebb(Z_j(t)\partial_t Z_i(t)) = \frac{1}{2}\Ebb(Y_iY_j-X_iX_j)$ for $j\in [n]$ indicates the monotonicity of $\partial_i f(\cdot)$ in its $j$-th coordinate. We can thus apply the positive association of Gaussian principle, stated below, to conclude the proof. 
\begin{lemma}\label{lem:positive-assoc}
    If $(Z,Z')\in \Rbb^{n+1}$ is a centered Gaussian vector in which $Z'$ is a scalar satisfying $\forall i\in [n] : \Ebb(Z_iZ')\geq 0$ and $h:\Rbb^n\rightarrow\Rbb$ has non-negative first derivative, then, $\Ebb(h(Z)Z')\geq 0$. 
\end{lemma}
To show Lemma \ref{lem:positive-assoc}, we write the Gaussian regression of $Z_i$ to $Z'$ for all $i\in [n]$ : $Z_i = \alpha_i Z' + W_i$, where $\alpha_i = \Ebb(Z_iZ')/\Ebb(Z'^2)\geq 0$, and where $W_i$ is a centered Gaussian random variable that is independent of $Z'$. Then, introducing $Z(u) = u\alpha_i Z'+W_i$, we note that $\zeta(u) = \Ebb(h(Z(u))Z')$ satisfies $\zeta(0)=0$ and $\zeta'(u) = \sum_{i\in [n]} \alpha_i \Ebb(\partial_i h(Z(u))Z'^2)\geq 0$, and, therefore, $\zeta(1) = \Ebb(h(Z)Z')\geq 0$.      
\end{proof}

\begin{remark}[Reference to detailed proofs] The proof of Theorem \ref{th:gordon} from Theorem \ref{th:gordon-orignal} can be found in the Appendix G.3 of \cite{miolane2021distribution}. The proof of Theorem \ref{th:gordon-orignal} from first principles can be found in the original work of Gordon \cite{gordon1985some}, but the more direct proof outlined here---due to Kahane \cite{kahane1986inegalite}---is available in the Chapter 3 of the book of Talagrand and Ledoux \cite{ledoux2013probability} (p.75). 
\end{remark}

\addcontentsline{toc}{section}{References}

\bibliography{biblio}

\end{document}